\documentclass[a4paper,11pt,leqno]{article}
\usepackage{amssymb, amsmath, amsthm, graphicx}

\newtheorem{theorem}{Theorem}[section]
\newtheorem{lem}[theorem]{Lemma}
\newtheorem{cor}[theorem]{Corollary}
\newtheorem{prop}[theorem]{Proposition}

\newtheorem{rem}{Remark}[section]
\newtheorem{exa}{Example}[section]
\numberwithin{equation}{section}

\renewcommand{\a}{\alpha}
\renewcommand{\b}{\beta}
\newcommand{\e}{\varepsilon}
\newcommand{\de}{\delta}
\newcommand{\fa}{\varphi}
\newcommand{\ga}{\gamma}
\renewcommand{\k}{\kappa}
\newcommand{\la}{\lambda}

\newcommand{\si}{\sigma}

\renewcommand{\t}{\tau}
\newcommand{\om}{\omega}
\newcommand{\De}{\Delta}
\newcommand{\Ga}{\Gamma}
\newcommand{\La}{\Lambda}

\newcommand{\lan}{\langle}
\newcommand{\ran}{\rangle}

\def\R{{\mathbb{R}}}
\def\N{{\mathbb{N}}}
\def\Z{{\mathbb{Z}}}
\def\T{{\mathbb{T}}}

\allowdisplaybreaks

\title{Motion by mean curvature from Glauber-Kawasaki dynamics}

\author{Tadahisa Funaki$\,^{1)}$ and Kenkichi Tsunoda$\,^{2)}$}
\date{\today}

\begin{document}
\maketitle

\begin{abstract}
\noindent
We study the hydrodynamic scaling limit for the Glauber-Kawasaki dynamics.
It is known that, if the Kawasaki part is speeded up in a diffusive space-time
scaling, one can derive the Allen-Cahn equation which is a kind of the reaction-diffusion
equation in the limit.  This paper concerns the scaling that the Glauber part, 
which governs the creation and annihilation of particles, is also speeded up
but slower than the Kawasaki part.  Under such scaling, we derive directly from 
the particle system the motion by mean curvature for the interfaces separating 
sparse and dense regions of particles as a combination of the hydrodynamic and sharp
interface limits.

\footnote{
\hskip -6mm 
${}^{1)}$ Department of Mathematics, School of Fundamental Science and Engineering,  Waseda University,
3-4-1 Okubo, Shinjuku-ku, Tokyo 169-8555, Japan. e-mail: funaki@ms.u-tokyo.ac.jp \\
${}^{2)}$ Department of Mathematics, Osaka University, 1-1, Machikaneyama-cho, Toyonaka, Osaka
560-0043, Japan. e-mail: k-tsunoda@math.sci.osaka-u.ac.jp}
\footnote{
\hskip -6mm
Abbreviated title $($running head$)$: MMC from Glauber-Kawasaki dynamics}
\footnote{
\hskip -6mm
2010MSC: Primary 60K35 ; Secondary 82C22, 74A50.}
\footnote{
\hskip -6mm
The author$^{1)}$ is supported in part by JSPS KAKENHI, 
Grant-in-Aid for Scientific Researches (S) 16H06338,
(A) 18H03672, 17H01093, 17H01097 and (B) 16KT0024, 26287014.
The author$^{2)}$ is supported in part by JSPS KAKENHI, Grant-in-Aid for
Early-Career Scientists 18K13426.}
\end{abstract}

\section{Introduction}

In this paper, we consider the Glauber-Kawasaki dynamics, that is
the simple exclusion process with an additional effect of
creation and annihilation of particles, on a $d$-dimensional 
periodic square lattice of size $N$ with $d\ge 2$ and study its hydrodynamic
behavior.  We introduce the diffusive space-time scaling for the Kawasaki part.
Namely, the time scale of particles performing random walks with exclusion
rule is speeded up by $N^2$.  It is known that, if the time scale 
of the Glauber part stays at $O(1)$, one can derive the reaction-diffusion
equation in the limit as $N\to\infty$.

This paper discusses the scaling under which
the Glauber part is also speeded up by the factor $K=K(N)$,
which is at the mesoscopic level.  More precisely, we take $K$ 
such as $K\to\infty$ satisfying $K \le 
\text{const}\times (\log N)^{1/2}$, and shows that the system exhibits
the phase separation.  In other words,
if we choose the rates of creation and annihilation in a proper
way, then microscopically the whole region is
separated into two regions occupied by different phases called
sparse and dense phases, and the macroscopic interface separating
these two phases evolves according to the motion by mean curvature.

\subsection{Known result on hydrodynamic limit}  \label{sec:1.1}

Before introducing our model, we explain a classical result on the hydrodynamic
limit for the Glauber-Kawasaki dynamics in a different scaling from ours.
Let $\T_N^d :=(\Z/N\Z)^d = \{1,2,\ldots,N\}^d$ be the $d$-dimensional 
square lattice of size $N$ with periodic boundary condition.
The configuration space is denoted by $\mathcal{X}_N = \{0,1\}^{\T_N^d}$ and
its element is described by $\eta=\{\eta_x\}_{x\in \T_N^d}$.
In this subsection, we discuss the dynamics with
the generator given by $L_N = N^2L_K+L_G$, where
\begin{align*}
L_Kf (\eta)   
&  = \frac12 \sum_{x, y\in\T_N^d:|x-y|=1}%
\left\{  f\left(  \eta^{x,y}\right)  -f\left(\eta\right)  \right\},  \\
L_Gf (\eta)  
&  =\sum_{x\in\T_N^d} c_x(\eta)
  \left\{  f\left(  \eta^x \right)  -f\left(  \eta\right)  \right\},
\end{align*}
for a function $f$ on $\mathcal{X}_N$.  The configurations
$\eta^{x,y}$ and $\eta^x\in \mathcal{X}_N$ are defined from 
$\eta\in \mathcal{X}_N$ as 
\begin{align*}
(\eta^{x,y})_z = \begin{cases}
     \eta_y  & \text{if $z=x$}, \\
     \eta_x  & \text{if $z=y$}, \\
     \eta_z    & \text{if $z\neq x,y$},
\end{cases}
\quad
(\eta^x)_z = \begin{cases}
     1-\eta_z  & \text{if $z=x$}, \\
     \eta_z  & \text{if $z\neq x$}. 
\end{cases}
\end{align*}
The flip rate $c(\eta)\equiv c_0(\eta)$ in the Glauber part is a nonnegative
local function on $\mathcal{X}:=\{0,1\}^{\Z^d}$ (regarded as that on
$\mathcal{X}_N$ for $N$ large enough), $c_x(\eta) = c(\t_x\eta)$ and $\t_x$ is the 
translation acting on $\mathcal{X}$ or $\mathcal{X}_N$ defined
by $(\t_x\eta)_z =\eta_{z+x}, z\in \Z^d$ or $\T_N^d$.
In fact, $c(\eta)$ has the following form:
\begin{equation}  \label{eq:c+-}
c(\eta) = c^+(\eta)(1-\eta_0)+c^-(\eta) \eta_0,
\end{equation}
where $c^+(\eta)$ and $c^-(\eta)$ represent the rates of creation 
and annihilation of a particle at $x=0$, respectively, and both are 
local functions which do not depend on the occupation variable $\eta_0$.

Let $\eta^N(t)=\{\eta_x^N(t)\}_{x\in \T_N^d}$ be the Markov process
on  $\mathcal{X}_N$ generated by $L_N$.  The macroscopically scaled
empirical measure on $\T^d$, that is $[0,1)^d$ with the periodic boundary,
associated with a configuration $\eta\in \mathcal{X}_N$ is defined by
\begin{align}  \notag
&\a^N(dv;\eta) = \frac1{N^d} \sum_{x\in \T_N^d}
\eta_x \de_{x/N}(dv),\quad v \in \T^d, 
\intertext{and we set}
& \a^N(t,dv) = \a^N(dv;\eta^N(t)), \quad t\ge 0.
\label{eq:aNt}
\end{align}

Then, it is known that the empirical measure
$\a^N(t,dv)$ converges to $\rho(t,v)dv$ as $N\to\infty$
in probability (multiplying a test function on $\T^d$)
if this holds at $t=0$.  Here, $\rho(t,v)$ is a unique weak solution
of the reaction-diffusion equation
\begin{equation}  \label{eq:RD}
\partial_t\rho= \De\rho+f(\rho), \quad v\in \T^d,
\end{equation}
with the given initial value $\rho(0)$, $dv$ is the Lebesgue measure on $\T^d$ and
\begin{equation}  \label{eq:4.f}
f(\rho) = E^{\nu_\rho}[(1-2\eta_0) c(\eta)],
\end{equation}
where $\nu_\rho$ is the Bernoulli measure on $\Z^d$ with mean $\rho\in [0,1]$.
This was shown by De Masi et al. \cite{DFL}; see also 
\cite{DP} and \cite{FLT} for further developments in the Glauber-Kawasaki dynamics.
We use the same letter $f$ for functions on $\mathcal{X}_N$ and the reaction
term defined by \eqref{eq:4.f}, but these should be clearly distinguished.

From \eqref{eq:c+-}, the reaction term can be rewritten as
\begin{align*}
f(\rho) & =  E^{\nu_\rho}[c^+(\eta) (1-\eta_0) - c^-(\eta)\eta_0] \\
& = c^+(\rho)(1-\rho) - c^-(\rho) \rho,
\end{align*}
if $c^\pm$ are given as the finite sum of the form:
\begin{equation}  \label{eq:cpm}
c^\pm(\eta) = \sum_{0\notin \La\Subset \Z^d} c_\La^\pm \prod_{x\in\La}\eta_x,
\end{equation}
with some constants $c_\La^\pm\in\R$. 
Note that $c^\pm(\rho) := E^{\nu_\rho}[c^\pm(\eta)]$ are equal to \eqref{eq:cpm}
with $\eta_x$ replaced by $\rho$.
We give an example of the flip rate $c(\eta)$ and the corresponding reaction
term $f(\rho)$ determined by \eqref{eq:4.f}.

\begin{exa}  \label{exa:4.1}
Consider $c^\pm(\eta)$ in \eqref{eq:c+-} of the form
\begin{equation}  \label{eq:1.cpm}
\begin{aligned}
& c^+(\eta) = a^+ \eta_{n_1}\eta_{n_2}+ b^+ \eta_{n_1} + c^+ >0, \\
& c^-(\eta) = a^- \eta_{n_1}\eta_{n_2}+ b^- \eta_{n_1} + c^- >0,
\end{aligned}
\end{equation}
with $a^\pm, b^\pm, c^\pm\in \R$ and $n_1, n_2 \in \Z^d$ such that three points $\{n_1, n_2, 0\}$ 
are different.  Then,
\begin{align}  \label{eq:Ex11f}
f(\rho) = -(a^++a^-)\rho^3 + (a^+-b^+-b^-)\rho^2 +(b^+-c^+-c^-)\rho+c^+.
\end{align}
In particular, under a suitable choice of six constants $a^\pm, b^\pm, c^\pm$, one can have
\begin{align}  \label{eq:4.bistable}
f(\rho)  = -C(\rho-\a_1)(\rho-\a_*)(\rho-\a_2),
\end{align}
with some $C>0$, $0<\a_1<\a_* < \a_2<1$ satisfying $\a_1+\a_2=2\a_*$;
see the example in Section 8 of \cite{FLT} with $\a_*=1/2$ given in $1$-dimensional setting.
Namely, $f(\rho)$ is bistable with stable points $\rho= \a_1, \a_2$ and unstable point
$\a_*$, and satisfies the balance condition $\int_{\a_1}^{\a_2}f(\rho)d\rho=0$.
\end{exa}

For the reaction term $f$ of the form \eqref{eq:4.bistable}, the equation \eqref{eq:RD}
considered on $\R$ instead of $\T^d$ admits a traveling wave solution 
which connects two different stable points $\a_1, \a_2$, and its speed is $0$
due to the balance condition.  The traveling wave solution with speed $0$
is called a standing wave.  See Section \ref{sec:4.2-a} for details.

\subsection{Our model and main result}
\label{sec:4.2}

The model we concern in this paper is the Glauber-Kawasaki dynamics
$\eta^N(t)=\{\eta_x^N(t)\}_{x\in \T_N^d}$ 
on $\mathcal X_N$ with the generator $L_N = N^2 L_K + K L_G$ with another
scaling parameter $K>0$.
The parameter $K$ depends on $N$ as $K=K(N)$ and tends to
$\infty$ as $N\to\infty$.  

If we fix $K$ so as to be independent of $N$, then, as we saw in Section \ref{sec:1.1}, we obtain the 
reaction-diffusion equation for $\rho\equiv \rho^K(t,v)$
in the hydrodynamic limit:
\begin{equation}  \label{eq:RD-2}
\partial_t\rho= \De\rho+Kf(\rho), \quad v\in \T^d.
\end{equation}
The partial differential equation (PDE) \eqref{eq:RD-2} 
with the large reaction term $Kf$, which is
bistable and satisfies the balance condition as in Example
\ref{exa:4.1}, is called the Allen-Cahn equation.  It is known
that as $K\to \infty$ the Allen-Cahn equation leads to
the motion by mean curvature; see Section \ref{sec:4.4}.
Our goal is to derive it directly from the particle system.  

For our main theorem, we assume the
following five conditions on the creation and annihilation rates
$c^\pm(\eta)$ and the mean $u^N(0,x) = E[\eta_x^N(0)]$,
$x\in \T_N^d$ of the initial distribution of our process.
\begin{itemize}
\item[(A1)] 
$c^\pm(\eta)$ have the form \eqref{eq:1.cpm} with
$n_1, n_2 \in \Z^d$, both of which have at least one positive
components, and three points $\{n_1,n_2,0\}$ are different.
\item[(A2)]
The corresponding $f$ defined by \eqref{eq:4.f} or equivalently by
\eqref{eq:Ex11f} is bistable, that is, $f$ has exactly three zeros 
$0<\a_1<\a_*<\a_2<1$ and $f'(\a_1)<0$, $f'(\a_2)<0$ hold, and it
satisfies the balance condition $\int_{\a_1}^{\a_2}f(\rho) d\rho=0$.
\item[(A3)] $c^+(u)$ is increasing and $c^-(u)$ is decreasing in 
$u=\{u_{n_k}\}_{k=1}^2 \in [0,1]^2$ under the partial order $u\ge v$
defined by $u_{n_k}\ge v_{n_k}$ for $k=1,2$, where $c^\pm(u)$ are 
defined by \eqref{eq:1.cpm} with $\eta_{n_k}$ replaced by $u_{n_k}$.
\item[(A4)] $\|\nabla^N u^N(0,x)\| \le C_0 K \, \big(\!=C_0K(N)\big)$ 
for some $C_0>0$, 
where $\nabla^N u(x) := \{N(u(x+e_i)-u(x))\}_{i=1}^d$ with the unit vectors
 $e_i\in \Z^d$ of the direction $i$ and $\|\cdot\|$ stands for the standard Euclidean norm of $\R^d$.
\item[(A5)] $u^N(0,v), v\in \T^d$ defined by \eqref{eq:2.uN-m} from
$u^N(0,x)$ satisfies the bound \eqref{eq:3.3-com} at $t=0$.
\end{itemize}

The condition (A5) implies that a smooth hypersurface $\Ga_0$ in $\T^d$ without
boundary exists and $u^N(0,x)$ converges to $\chi_{\Ga_0}(v)$ weakly in $L^2(\T^d)$
as $N\to\infty$, see \eqref{eq:4.5-A} taking $t=0$.   We denote for a closed hypersurface $\Ga$ 
which separates $\T^d$ into two disjoint regions,
\begin{equation}  \label{eq:1.chiga}
\chi_{\Ga}(v) := \left\{
\begin{aligned}
\a_1,& \quad \text{ for } v \text{ on one side of } \Ga, \\
\a_2,& \quad \text{ for } v \text{ on the other side of } \Ga.
\end{aligned}
\right.
\end{equation}
It is known that a smooth family of closed hypersurfaces $\{\Ga_t\}_{t\in[0,T]}$ in 
$\T^d$, which starts from $\Ga_0$ and evolves being governed by the motion by mean
curvature \eqref{eq:4.MMC}, exists until some time $T>0$; recall $d \ge 2$ and
see the beginning of Section \ref{sec:4.4} for details.  Note that the sides of $\Ga_t$ in
\eqref{eq:1.chiga} with $\Ga=\Ga_t$ is kept under the time evolution and
determined continuously from those of $\Ga_0$.  We need the smoothness 
of $\Ga_t$ to construct super and sub solutions of the discretized hydrodynamic
equation in Theorem \ref{thm:3.4}.

Let $\mu_0^N$ be the distribution of $\eta^N(0)$ on $\mathcal{X}_N$
and let $\nu_0^N$ be the Bernoulli measure on $\mathcal{X}_N$ with 
mean $u^N(0) = \{u^N(0,x)\}_{x\in \T_N^d}$.  Recall that $u^N(0,x)$ is the mean of 
$\eta_x$ under $\mu_0^N$ for each $ x\in \T_N^d$. Our another condition
with $\de>0$ is the following.

\begin{itemize}
\item[(A6)$_\de$] The relative entropy at $t=0$ defined by
\eqref{eq:2.RE} behaves as $H(\mu_0^N|\nu_0^N) = O(N^{d-\de_0})$ 
as $N\to\infty$ with some $\de_0>0$ and $K=K(N) \to\infty$ satisfies
$1\le K(N) \le \de (\log N)^{1/2}$.
\end{itemize}

The main result of this paper is now stated as follows.
Recall that $\a^N(t)$ is defined by \eqref{eq:aNt}.

\begin{theorem} \label{thm:4.10}
Assume the six conditions {\rm (A1)--(A6)}$_\de$ with $\de>0$ small enough 
chosen depending on $T$.  Then, we have
\begin{equation}  \label{eq:thm1.1}
\lim_{N\to\infty} P\left( \left|\lan \a^N(t),\fa\ran - \lan \chi_{\Ga_t},\fa\ran \right|>\e\right) =0,
\quad t\in [0,T],
\end{equation}
for every $\e>0$ and $\fa\in C^\infty(\T^d)$, where
$\lan\a,\fa\ran$ or $\lan\chi_\Ga,\fa\ran$ 
denote the integrals on $\T^d$ of $\fa$ with respect to the measures $\a$ or 
$\chi_\Ga(v)dv$, respectively.
\end{theorem}

The proof of Theorem \ref{thm:4.10} consists of two parts, 
that is, the probabilistic part in Sections \ref{sec:4.3} and \ref{sec:3}, and
the PDE part in Section \ref{sec:4.4}.
In the probabilistic part, we apply the relative entropy method of 
Jara and Menezes \cite{JM1}, \cite{JM2}, which is in a sense 
a combination of the methods due to Guo et al.\ 
\cite{GPV} and Yau \cite{Y}.  In the PDE part, we show the convergence
of solutions of the discretized hydrodynamic equation \eqref{eq:HD-discre} 
with the limit governed by the motion by mean curvature.
More precise rate of convergence in \eqref{eq:thm1.1} is given in
Remark \ref{rem:3.1}.

We give some explanation for our conditions.
If we take $a^+=32, b^+=0, c^+=3, a^-=0, b^-=-16, c^-=19$ in Example \ref{exa:4.1}, we have
$c^+(u) =32u_{n_1}u_{n_2}+3$, $c^-(u) = -16u_{n_1}+19$ and $f(\rho)$ has the
form \eqref{eq:4.bistable} with $C=32$, $\a_1=1/4$, $\a_*=1/2$ and 
$\a_2=3/4$ so that the conditions (A1)--(A3) are satisfied.
For simplicity, we discuss in this paper $c^\pm(\eta)$ of the form \eqref{eq:1.cpm}
only, however one can generalize our result to more general $c^\pm(\eta)$
given as in \eqref{eq:cpm}.  The corresponding $f$ may have several zeros,
but we may restrict our arguments in the PDE part to a subinterval of $[0,1]$,
on which the conditions (A2) and (A3) are satisfied.
The entropy condition in (A6)$_\de$ is satisfied, for example, if 
$d\mu_0^N = g_Nd\nu_0^N$ and $\log \|g_N\|_\infty \le C N^{d-\de_0}$
holds for some $C>0$.

In the probabilistic part, we only need the following condition weaker
than (A5).

\begin{itemize}
\item[(A7)] $u_- \le u^N(0,x) \le u_+$ for some $0<u_-<u_+<1$.
\end{itemize}

For convenience, we take $u_\pm$ such that $0<u_-<\a_1<\a_2<u_+<1$
by making $u_-$ smaller and $u_+$ larger if necessary; 
see the comments given below Theorem \ref{prop:4.4-2}.  The condition
(A7) with this choice of $u_\pm$ is called (A7)$'$.
Under this choice of $u_\pm$, the condition (A3) can be weakened and it is 
sufficient if it holds for $u\in [u_-,u_+]^2$.  
The conditions (A1), (A4), (A6)$_\de$, (A7) are used in the probabilistic part, while
(A2), (A3), (A5) are used in the PDE part.  To be precise, (A2), (A3) are
used also in the probabilistic part but in a less important way; see
the comments below Theorem \ref{thm:EstHent}.

The derivation of the motion by mean curvature and the related 
problems of pattern formation in interacting particle systems were 
discussed by Spohn \cite{S93} rather heuristically, and by
De Masi et al.\ \cite{DPPV94}, Katsoulakis and Souganidis 
\cite{KS94}, Giacomin \cite{G95} for Glauber-Kawasaki dynamics. 
De Masi et al.\ \cite{DOPT93}, 
\cite{DOPT94}, Katsoulakis and Souganidis \cite{KS95} studied
Glauber dynamics with Kac type long range mean field interaction.
Related problems are discussed by
Caputo et al.\ \cite{CMST}, \cite{CMT}.
Similar idea is used in Hern\'andez et al.\ \cite{HJV}
to derive the fast diffusion equation from zero-range processes.
Bertini et al.\ \cite{BBP} discussed from the viewpoint
of large deviation functionals.

In particular, the results of \cite{KS94} are close to ours.  
They consider the Glauber-Kawasaki dynamics with generator
$\la^{-2}(\e^{-2}L_K+L_G)$ under the spatial scaling $(\la\e)^{-1}$,
where $\la=\la(\e) \, (\downarrow 0)$ should satisfy the condition
$\lim_{\e\downarrow 0} \e^{-\zeta^*}\la(\e) = \infty$ with some
$\zeta^*>0$.  If we write $N=(\la\e)^{-1}$ as in our case, the
generator becomes $N^2L_K+\la^{-2}L_G$ so that $\la^{-2}$ plays a role 
similar to our $K=K(N)$.  They analyze the limit of correlation functions.

On the other hand, our analysis makes it possible to study the limit
of the empirical measures, which is more natural in the study of the
hydrodynamic limit, under a milder assumption on the initial distribution
$\mu_0^N$.  Moreover, we believe that our relative entropy method
has an advantage to work for a wide class of models in parallel.
Furthermore, this method is applicable to study the fast-reaction limit 
for two-component Kawasaki dynamics, which leads to the two-phase 
Stefan problem, see \cite{DFPV}.

Finally, we make a brief comment on the case that $f$ is unbalanced:
$\int_{\a_1}^{\a_2}f(\rho)d\rho\not=0$.  For such $f$, the proper time scale is
shorter and turns out to be $K^{-1/2}t$, so that the equation \eqref{eq:RD-2}
is rescaled as
\begin{equation*} 
\partial_t\rho= K^{-1/2}\De\rho+K^{1/2}f(\rho), \quad v\in \T^d.
\end{equation*}
It is known that this equation exhibits a different behavior in the sharp
interface limit as $K\to\infty$, see p.\ 95 of \cite{F16}.  The present paper
does not discuss this case.

\section{Relative entropy method}  \label{sec:4.3}

We start the probabilistic part by formulating Theorem \ref{thm:EstHent}.
This gives an estimate on the relative entropy of our system with
respect to the local equilibria and implies the weak law of large numbers 
\eqref{eq:4.18} as its consequence.  We compute the time derivative of the relative
entropy to give the proof of Theorem \ref{thm:EstHent}.
In Sections 2 and 3, it is unnecessary to assume $d\ge 2$, so that we discuss for all
$d\ge 1$ including $d=1$.

\subsection{The entropy estimate}

From \eqref{eq:c+-}, the flip rate $c_x(\eta) \equiv c(\t_x\eta)$
of the Glauber part has the form
\begin{equation} \label{eq:G-rate}
c_x(\eta) = c_x^+(\eta)(1-\eta_x) + c_x^-(\eta) \eta_x,
\end{equation}
where $c_x^\pm(\eta)= c^\pm(\t_x\eta)$ with $c^\pm(\eta)$ of the form \eqref{eq:cpm}.
Let $u^N(t) = \{u^N(t,x)\in [0,1]\}_{x\in \T_N^d}$
be the solution of the discretized hydrodynamic equation:
\begin{equation} \label{eq:HD-discre}
\partial_t u^N(t,x) = \De^N u^N(t,x) + K f^N(x,u^N(t)), \quad x\in \T_N^d,
\end{equation}
where $f^N(x,u)$ is defined by
\begin{equation} \label{eq:f^N}
f^N(x,u) = (1-u_x)c_x^+(u)-u_xc_x^-(u),
\end{equation}
for $u\equiv \{u_x = u(x)\}_{x\in \T_N^d}$ and
$$
\De^N u(x) := N^2 \sum_{y\in \T_N^d:|y-x|=1} \left(u(y) - u(x)\right)
=N^2 \sum_{i=1}^d \left(u(x+e_i) + u(x-e_i) - 2u(x)\right).
$$
Note that $c_x^\pm(u):= 
E^{\nu_u}[c^\pm_x(\eta)]$ are given by \eqref{eq:cpm} with $\eta_x$ replaced by
$u_x$ and $\nu_u$ is the Bernoulli measure with non-constant mean $u=u(\cdot)$.
In the following, we assume that $c^\pm(\eta)$ have the form \eqref{eq:1.cpm} and,
in this case, we have
\begin{align*}
f^N(x,u) =& -(a^++a^-)u_xu_{x+n_1}u_{x+n_2}  +a^+u_{x+n_1}u_{x+n_2} \\
&\quad  -(b^++b^-)u_xu_{x+n_1} -(c^++c^-) u_x+b^+ u_{x+n_1} +c^+.
\end{align*}

Let $\mu$ and $\nu$ be two probability measures on $\mathcal{X}_N$.
We define the relative entropy of $\mu$ with respect to $\nu$ by
\begin{equation}  \label{eq:2.RE}
H(\mu|\nu) :=
\int_{\mathcal{X}_N} \frac{d\mu}{d\nu} \log \frac{d\mu}{d\nu} d\nu
\; (\ge 0), 
\end{equation}
if $\mu$ is absolutely continuous with respect to $\nu$,
$H(\mu|\nu) :=\infty$, otherwise.
Let $\mu_t^N$ be the distribution of $\eta^N(t)$ on $\mathcal{X}_N$
and let $\nu_t^N$ be the Bernoulli measure on $\mathcal{X}_N$ with mean
$u^N(t) = \{u^N(t,x)\}_{x\in \T_N^d}$.  The following result plays an essential role
to prove Theorem \ref{thm:4.10}.

\begin{theorem} \label{thm:EstHent}
We assume the conditions {\rm (A1)--(A4)} and {\rm (A7)}$'$.
Then, if {\rm (A6)}$_\de$ holds with small enough $\de>0$, we have
$$
H(\mu_t^N|\nu_t^N) = o(N^d), \quad t\in [0,T],
$$
as $N\to\infty$. 
The constant $\de>0$ depends on $T>0$.
\end{theorem}

Note that the condition (A7)$'$, i.e.\ (A7) with an additional condition 
on the choice of $u_\pm$, combined with the
comparison theorem implies that the solution $u^N(t,x)$ of the 
discretized hydrodynamic equation \eqref{eq:HD-discre} satisfies that 
$u_- \le u^N(t,x) \le u_+$ for all $t\in [0,T]$ and $x\in \T_N^d$;
see the comments given below Theorem \ref{prop:4.4-2}.
The conditions (A2) and (A3) are used only to show this bound for $u^N(t,x)$.

\subsection{Consequence of Theorem \ref{thm:EstHent}}

We define the macroscopic function $u^N(t,v), v\in\T^d$
associated with the microscopic function $u^N(t,x), x\in \T_N^d$ as
a step function
\begin{equation}  \label{eq:2.uN-m}
u^N(t,v) = \sum_{x\in\T_N^d} u^N(t,x) 1_{B(\frac xN,\frac 1N)}(v), 
\quad v \in \T^d,
\end{equation}
where $B(\frac{x}N,\frac1N) = \prod_{i=1}^d [\frac{x_i}N-\frac1{2N}, 
\frac{x_i}N+\frac1{2N})$ is the box with center $x/N$ and side 
length $1/N$.  Then the entropy inequality (see Proposition A1.8.2 of \cite{KL}
or Section 3.2.3 of \cite{F18})
$$
\mu_t^N(\mathcal A) \le \frac{\log 2 + H(\mu_t^N|\nu_t^N)}{\log \{1+1/\nu_t^N(\mathcal A)\}},
\quad \mathcal A\subset \mathcal X_N,
$$
combined with Theorem \ref{thm:EstHent} and Proposition \ref{prop:4.4-1} stated below shows that
\begin{equation}  \label{eq:4.18}
\lim_{N\to\infty}\mu_t^N ({\cal A}_{N,t}^\e)=0,
\end{equation}
for every $\e>0$, where
\[ 
{\cal A}_{N,t}^\e\equiv {\cal A}_{N,t,\fa}^\e := \left\{\eta\in{\cal X}_N \,; ~\left|
\lan \a^N,\fa\ran   - \lan u^N(t,\cdot),\fa\ran\right| > \e\right\},
\quad \fa\in C^\infty(\T^d).
\]

\begin{prop}  \label{prop:4.4-1}
There exists $C=C_\e>0$ such that
$$
\nu_t^N ({\cal A}_{N,t}^\e) \le e^{-CN^d}.
$$
\end{prop}

\begin{proof}
Set and observe
$$
X:= \lan \a^N,\fa\ran   - \lan u^N(t,\cdot),\fa\ran
= \frac1{N^d} \sum_{x\in\T_N^d} \left\{\eta_x-u^N(t,x)\right\}\fa(\frac{x}N) +o(1),
$$
as $N\to\infty$.  Then, we have
\begin{align*}
\nu_t^N({\cal A}_{N,t}^\e)
& \le e^{-\ga \e N^d} E^{\nu_t^N}[ e^{\ga N^d |X|}] \\
& \le e^{-\ga \e N^d} \left\{E^{\nu_t^N}[ e^{\ga N^d X}] + E^{\nu_t^N}[ e^{-\ga N^d X}] \right\},
\end{align*}
for every $\ga>0$, where we used the elementary inequality $e^{|x|}\le e^x+e^{-x}$
to obtain the second inequality. By the independence of $\{\eta_x\}_{x\in\T_N^d}$ under $\nu_t^N$,
the expectations inside the last braces can be written as
\begin{align*}
E^{\nu_t^N}[ e^{\pm \ga N^d X}]
& = \prod_{x\in\T_N^d} E^{\nu_t^N}[ e^{\pm \ga \{\eta_x-u_x\}\fa_x+o(1)}] \\
&= \prod_{x\in \T_N^d} \left\{ e^{\pm \ga(1-u_x)\fa_x} u_x + e^{\mp \ga u_x \fa_x}(1-u_x)\right\}
+o(1),
\end{align*}
where $u_x = u^N(t,x)$ and $\fa_x = \fa(x/N)$.  Applying the Taylor's formula at $\ga=0$, we see
$$
\left| e^{\pm \ga(1-u_x)\fa_x} u_x + e^{\mp \ga u_x \fa_x}(1-u_x) - 1 \right|
\le  C\ga^2, \quad C=C_{\|\fa\|_\infty},
$$
for $0<\ga\le 1$.  Thus we obtain
$$
\nu_t^N({\cal A}_{N,t}^\e)\le e^{-\ga \e N^d + C\ga^2 N^d},
$$
for $\ga>0$ sufficiently small.  This shows the conclusion.
\end{proof}

\subsection{Time derivative of the relative entropy}
For a function $f$ on $\mathcal{X}_N$ and a measure $\nu$ on $\mathcal{X}_N$,
set 
\begin{equation}  \label{eq:4.DKL}
\mathcal{D}_N(f;\nu) = 2 N^2\mathcal{D}_K(f;\nu) + K\mathcal{D}_G(f;\nu),
\end{equation}
where
\begin{align*}
\mathcal{D}_K(f;\nu)& = \frac14 \sum_{x,y\in \T_N^d: |x-y|=1} \int_{\mathcal{X}_N} \{f(\eta^{x,y})-f(\eta)\}^2 d\nu, \\
\mathcal{D}_G(f;\nu)& = \sum_{x\in\T_N^d} \int_{\mathcal{X}_N} c_x(\eta)\{f(\eta^x)-f(\eta)\}^2 d\nu.
\end{align*}

Take a family of probability measures
$\{\nu_t\}_{t\ge0}$ on $\mathcal{X}_N$ differentiable in $t$ and a probability measure
$m$ on $\mathcal{X}_N$ as a reference measure, and set $\psi_t(\eta) := 
(d\nu_t/dm)(\eta)$.  Assume that these measures have full supports in 
$\mathcal{X}_N$.  We denote the adjoint of an operator $L$ on $L^2(m)$ by $L^{*,m}$ in general.
Then we have the following proposition called Yau's inequality; see Theorem 4.2 of \cite{F18}
or Lemma A.1 of \cite{JM2} for the proof.

\begin{prop} \label{thm:4.2}
\begin{equation} \label{eq:dH}
\frac{d}{dt} H(\mu_t^N|\nu_t) \le - \mathcal{D}_N\left(\sqrt{\frac{d\mu_t^N}{d\nu_t}}; \nu_t\right) 
+ \int_{\mathcal{X}_N} (L^{*,\nu_t} {\bf 1} - \partial_t \log \psi_t) d\mu_t^N,
\end{equation}
where ${\bf 1}$ stands for the constant function ${\bf 1}(\eta)\equiv 1$, $\eta\in\mathcal X_N$.
\end{prop}

We apply Proposition \ref{thm:4.2} with $\nu_t=\nu_t^N$ to prove Theorem \ref{thm:EstHent}.

\subsection{Computation of $L_N^{*,\nu_t^N}{\bf 1} - \partial_t \log\psi_t$}

We compute the integrand of the second term in the right hand side of \eqref{eq:dH}.
Similar computations are made in the proofs of Lemma 3.1
of \cite{FUY}, Appendix A.3 of \cite{JM2} and Lemmas 4.4--4.6 of \cite{F18}.
We introduce the centered variable
$\bar\eta_x$ and the normalized centered variable $\om_x$ of $\eta_x$ under the
Bernoulli measure with mean $u(\cdot) = \{u_x\}_{x\in\T_N^d}$ as follows:
\begin{align*}
\bar\eta_x = \eta_x-u_x \quad \text{ and } \quad
\om_x = \frac{\bar\eta_x}{\chi(u_x)},
\end{align*}
where $\chi(\rho) = \rho(1-\rho)$, $\rho\in[0,1]$.  
We first compute the contribution of the Kawasaki part.

\begin{lem}  \label{lem:4.4-q}
Let $\nu= \nu_{u(\cdot)}$ be a Bernoulli measure on $\mathcal{X}_N$
with mean $u(\cdot) = \{u_x\}_{x\in\T_N^d}$.  Then, we have
\begin{equation} \label{eq:3.3-b}
L_K^{*,\nu}{\bf 1} =  -\frac12 \sum_{x,y\in\T_N^d:|x-y|=1}  (u_y-u_x)^2\om_x \om_y
+ \sum_{x\in\T_N^d} (\De u)_x \, \om_x,
\end{equation}
where $(\De u)_x = \sum_{y\in\T_N^d:|y-x|=1} (u_y-u_x)$.
\end{lem}

\begin{proof}
Take a test function $f$ on $\mathcal{X}_N$ and compute
\begin{align*}
\int_{\mathcal{X}_N} L_K^{*,\nu}{\bf 1} \cdot f d\nu 
& = \int_{\mathcal{X}_N} L_K f d\nu
= \frac12 \sum_{\eta\in \mathcal{X}_N}  \sum_{x,y\in\T_N^d: |x-y|=1}
 \{f(\eta^{x,y})-f(\eta)\} \nu(\eta)  \\
& = \sum_{\eta\in \mathcal{X}_N} \sum_{x,y\in\T_N^d: |x-y|=1} f(\eta) 
\frac{u_y-u_x}{u_x(1-u_y)} {\bf 1}_{\{\eta_x=1,\eta_y=0\}} \nu(\eta),
\end{align*}
where ${\bf 1}_{\mathcal A}$ denotes the indicator function of a set $\mathcal A\subset \mathcal X_N$.
To obtain the second line, we have applied the change of variables $\eta^{x,y}\mapsto\eta$,
and then the identity
$$
\nu(\eta^{x,y}) =  \left\{\frac{(1-u_x)u_y}{u_x(1-u_y)} 1_{\{\eta_x=1,\eta_y=0\}}
+\frac{(1-u_y)u_x}{u_y(1-u_x)} 1_{\{\eta_y=1,\eta_x=0\}}
+ 1_{\{\eta_x=\eta_y\}} \right\} \nu(\eta),
$$
and finally the symmetry in $x$ and $y$.  Since one can rewrite
${\bf 1}_{\{\eta_x=1,\eta_y=0\}}$ as 
$\bar\eta_x(1-\bar\eta_y) -\bar\eta_x u_y - \bar\eta_y u_x + u_x(1 -u_y)$, we have
$$
L_K^{*,\nu}{\bf 1} = \sum_{x,y\in\T_N^d: |x-y|=1} \frac{u_y-u_x}{u_x(1-u_y)} 
\{ \bar\eta_x(1-\bar\eta_y) -\bar\eta_x u_y - \bar\eta_y u_x + u_x(1 -u_y)\}.
$$
However, the sum of the last term vanishes, while
the sum of the second and third terms is computed by exchanging the role of $x$ and $y$ in the
third term and in the end we obtain
\begin{equation} \label{eq:3.3-a}
L_K^{*,\nu}{\bf 1} = \sum_{x,y\in\T_N^d: |x-y|=1} \left\{\frac{u_y-u_x}{u_x(1-u_y)} \bar\eta_x(1-\bar\eta_y)
- \frac{(u_y-u_x)^2}{(1-u_x)u_x(1-u_y)} \bar\eta_x\right\}.
\end{equation}
The right hand side in \eqref{eq:3.3-a} can be further rewritten as
\begin{equation} \label{eq:3.3-b-c}
-\sum_{x,y\in\T_N^d: |x-y|=1} \frac{u_y-u_x}{u_x(1-u_y)} \bar\eta_x \bar\eta_y
+ \sum_{x\in\T_N^d} \frac{(\De u)_x}{\chi(u_x)} \bar\eta_x,
\end{equation}
by computing the coefficient of $\bar\eta_x$ in \eqref{eq:3.3-a} as
\begin{align*}
(u_y-u_x)\left\{\frac{1}{u_x(1-u_y)} - \frac{(u_y-u_x)}{(1-u_x)u_x(1-u_y)} \right\}
 = \frac{u_y-u_x}{\chi(u_x)},
\end{align*}
which gives $(\De u)_x$ by taking the sum in $y$.  Finally, the first term in \eqref{eq:3.3-b-c}
can be symmetrized in $x$ and $y$ and we obtain
\eqref{eq:3.3-b}.
\end{proof}

The following lemma is for the Glauber part.  Recall that the
flip rate $c_x(\eta)$ is given by \eqref{eq:G-rate} with $c^\pm(\eta)$
of the form \eqref{eq:cpm} in general.

\begin{lem}\label{lem:4.5}
The Bernoulli measure 
$\nu=\nu_{u(\cdot)}$ is the same as in Lemma \ref{lem:4.4-q}.  Then, we have
\begin{align} \label{eq:L_G1}
L_G^{*,\nu}{\bf 1} &= \sum_{x\in\T_N^d} \left(\frac{ c_x^+(\eta)}{u_x}-
\frac{c_x^-(\eta)}{1-u_x}\right)\bar\eta_x \\
& = F(\om,u) + \sum_{x\in\T_N^d} f^N(x,u) \om_x,
\notag
\end{align}
where $f^N(x,u)$ is given by \eqref{eq:f^N} and
$$
F(\om,u) = \sum_{x\in \T_N^d} \sum_{\La\Subset\Z^d: |\La|\ge 2} 
c_\La(u_{\cdot+x}) \prod_{y\in\La}\om_{y+x},
$$
with a finite sum in $\La$ with $|\La|\ge 2$ and some local functions $c_\La(u)$ of $u\, (=\{u_x\}_{x\in\Z^d})$
for each $\La$.  In particular, if $c^\pm(\eta)$ have the form \eqref{eq:1.cpm}, we have
\begin{align}  \label{eq:2.Fabc}
F(\om,u) = & \sum_{x\in \T_N^d} a(u_x,u_{x+n_1},u_{x+n_2}) \om_x\om_{x+n_1}
+ \sum_{x\in \T_N^d} b(u_x,u_{x+n_1},u_{x+n_2}) \om_x\om_{x+n_2} \\
& + \sum_{x\in \T_N^d} c(u_x,u_{x+n_1},u_{x+n_2}) \om_x\om_{x+n_1}\om_{x+n_2},  \notag
\end{align}
where $a, b, c$ are shift-invariant bounded functions of $u$ defined by
\begin{align*}
& a(u_x,u_{x+n_1},u_{x+n_2}) = \chi(u_{x+n_1}) [\{ a^+(1-u_x)-a^- u_x\}u_{x+n_2}+b^+(1-u_x)-b^- u_x], \\
& b(u_x,u_{x+n_1},u_{x+n_2}) = \chi(u_{x+n_2}) \{a^+(1-u_x)-a^- u_x\} u_{x+n_1}, \\
& c(u_x,u_{x+n_1},u_{x+n_2}) = \chi(u_{x+n_1}) \chi(u_{x+n_2}) \{a^+(1-u_x)-a^-u_x\},
\end{align*}
respectively.
\end{lem}

\begin{proof}
The first identity in \eqref{eq:L_G1} is shown by computing
$\int_{\mathcal{X}_N} L_G^{*,\nu}{\bf 1}\cdot f d\nu$ for a test function $f$
and applying the change of variables $\eta^x\mapsto \eta$ as in the proof of Lemma
\ref{lem:4.4-q}; note that
\begin{align*}
& c_x(\eta^x) = c_x^+(\eta)\eta_x + c_x^-(\eta)(1-\eta_x),
\intertext{and}
& \nu(\eta^x) = \left\{ \frac{1-u_x}{u_x}\eta_x + \frac{u_x}{1-u_x} (1-\eta_x)\right\}
\nu(\eta).
\end{align*}

To see the second identity in \eqref{eq:L_G1}, we recall \eqref{eq:cpm} and note that
$$
\prod_{y\in\La}\eta_y = \prod_{y\in\La}(\bar\eta_y+u_y)
= \sum_{\emptyset \not= A \subset \La} \left(\prod_{y\in \La\setminus A}u_y\right)
\prod_{y\in A} \bar\eta_y + \prod_{y\in \La}u_y,
$$
for $0\notin \La \Subset \Z^d$.  Therefore, we have
$$
\left(\frac{ c_x^+(\eta)}{u_x}-
\frac{c_x^-(\eta)}{1-u_x}\right)\bar\eta_x 
= (\text{the term containing two or more } \bar\eta\text{'s})
+ \left(\frac{ c_x^+(u)}{u_x}-
\frac{c_x^-(u)}{1-u_x}\right)\bar\eta_x.
$$
Since the last term is equal to $f^N(x,u) \om_x$,
 this shows the second identity with
\begin{align*}
F(\om,u) & = \sum_{x\in\T_N^d} \left\{
\left(\frac{ c_x^+(\eta)}{u_x}-
\frac{c_x^-(\eta)}{1-u_x}\right)
- \left(\frac{ c_x^+(u)}{u_x}-
\frac{c_x^-(u)}{1-u_x}\right) \right\} \bar\eta_x \\
& = \sum_{x\in\T_N^d} \{(1-u_x)(c_x^+(\eta)- c_x^+(u)) - u_x(c_x^-(\eta)- c_x^-(u)) \} \om_x.
\end{align*}
In particular, for  $c^\pm(\eta)=c_0^\pm(\eta)$ of the form \eqref{eq:1.cpm}, 
we have
\begin{align*}
c^+(\eta)- c^+(u) 
& = a^+(\eta_{n_1}\eta_{n_2} - u_{n_1}u_{n_2}) + b^+(\eta_{n_1} - u_{n_1})  \\
& = a^+(\bar\eta_{n_1} \bar\eta_{n_2} + u_{n_2}\bar\eta_{n_1} + u_{n_1} \bar\eta_{n_2})
+ b^+\bar\eta_{n_1}, \\
c^-(\eta)- c^-(u)
& = a^-(\bar\eta_{n_1} \bar\eta_{n_2} + u_{n_2}\bar\eta_{n_1} + u_{n_1} \bar\eta_{n_2})
+ b^-\bar\eta_{n_1}.
\end{align*}
This leads to the desired formula \eqref{eq:2.Fabc}.
\end{proof}

We have the following lemma for the last term in \eqref{eq:dH}.

\begin{lem}\label{lem:4.6}
Recalling that $\nu_t = \nu_{u(t,\cdot)}$, $u(t,\cdot)=\{u_x(t)\}_{x\in\T_N^d}$, 
is Bernoulli, we have
\begin{equation} \label{eq:logpsi}
\partial_t \log\psi_t = \sum_{x\in\T_N^d} \partial_t u_x(t)\om_{x,t},
\end{equation}
where $\om_{x,t} = \bar\eta_x/\chi(u_x(t))$. 
\end{lem}

\begin{proof}
The proof is straightforward.  In fact, we have by definition
$$
\psi_t(\eta) = \frac{\nu_t(\eta)}{m(\eta)}
= \frac{\prod_x \{u_x(t)\eta_x + (1-u_x(t))(1-\eta_x)\}}{m(\eta)},
$$
and therefore,
\begin{align*}
\partial_t \log\psi_t(\eta) & = \sum_x \frac{\partial_t u_x(t)(2\eta_x-1)}
{u_x(t)\eta_x + (1-u_x(t))(1-\eta_x)}  \\
& = \sum_x \left\{ \frac{\partial_t u_x(t)}{u_x(t)} {\bf 1}_{\{\eta_x=1\}}
-\frac{\partial_t u_x(t)}{1-u_x(t)} {\bf 1}_{\{\eta_x=0\}}\right\}.
\end{align*}
This shows the conclusion.
\end{proof}

The results obtained in Lemmas \ref{lem:4.4-q}, \ref{lem:4.5} and \ref{lem:4.6} are summarized in
the following corollary.  Note that the discretized hydrodynamic equation \eqref{eq:HD-discre} 
exactly cancels the first order term in $\om$.
Therefore only quadratic or higher order terms in $\om$ survive. We denote the solution
$u^N(t) = \{u^N(t,x)\}_{x\in\T_N^d}$ of \eqref{eq:HD-discre} simply by $u(t)=\{u_x(t)\}_{x\in\T_N^d}$.

\begin{cor}  \label{Cor:2.6}
We have
\begin{equation*}
L_N^{*,\nu_t^N}{\bf 1} - \partial_t \log\psi_t
= -\frac{N^2}2 \sum_{x,y\in \T_N^d:|x-y|=1} (u_y(t)-u_x(t))^2 \om_{x,t} \om_{y,t} 
+ K F(\om_t,u(t)),
\end{equation*}
where $\om_t = (\om_{x,t})$.  In particular, when $c^\pm(\eta)$ are given by \eqref{eq:1.cpm}, 
omitting to write the dependence on $t$, this is equal to
\begin{align}  \label{eq:cor26}
& -\frac{N^2}2 \sum_{x,y\in \T_N^d:|x-y|=1} (u_y-u_x)^2\om_x \om_y \\
& \qquad + K \left\{ \sum_{x\in\T_N^d} \tilde\om_x^{(a)}\om_{x+n_1}
+ \sum_{x\in\T_N^d} \tilde\om_x^{(b)} \om_{x+n_2} 
 + \sum_{x\in\T_N^d} \tilde\om_x^{(c)} \om_{x+n_1}\om_{x+n_2}\right\}, \notag \\
& \qquad =: V_1+V_a+V_b+V_c, \notag
\end{align}
where $\tilde\om_x^{(a)}$ stands for $a(u_x,u_{x+n_1},u_{x+n_2})\om_x$, and
$\tilde\om_x^{(b)}$ and $\tilde\om_x^{(c)}$ are defined similarly.
\end{cor}

\section{Proof of Theorem \ref{thm:EstHent}}  \label{sec:3}

We prove in this section Theorem \ref{thm:EstHent}.  In view of Proposition
\ref{thm:4.2} and Corollary \ref{Cor:2.6},
our goal is to estimate the following expectation under $\mu_t$ by the 
Dirichlet form $N^2\mathcal{D}_K(\sqrt{f_t};\nu_t)$ and the 
relative entropy $H(\mu_t|\nu_t)$
itself, where $f_t = d\mu_t/d\nu_t$ and $\mu_t = \mu_t^N$, $\nu_t = \nu_t^N$:
\begin{equation}  \label{eq:2.int}
\int_{\mathcal{X}_N} \left\{-\frac{N^2}2 \sum_{x,y\in \T_N^d:|x-y|=1} (u_y(t)-u_x(t))^2 \om_{x,t} \om_{y,t}
 + K F(\om_t,u(t))\right\} d \mu_t.
\end{equation}
Note that the condition (A7)$'$ implies
that $\chi(u_x(t))^{-1}=\chi(u^N(t,x))^{-1}$ appearing in the definition
of $\om_{x,t}$ is bounded; see the comments given below 
Theorem \ref{prop:4.4-2}.  From the condition (A4) combined with
Proposition \ref{prop:nabla-u} stated below,
the first term in \eqref{eq:2.int} can be treated similarly to the second, 
but with the front factor $K$ replaced by $K^2$; see Section \ref{sec:3.3} for details.

\subsection{Replacement by local sample average}

Recall that we assume $c^\pm(\eta)$ have the form \eqref{eq:1.cpm}
by the condition (A1)  so that
$F(\om,u)$ has the form \eqref{eq:2.Fabc}.  With this in mind,
recall the definition of $V_a$ defined in \eqref{eq:cor26}:
$$
V_a \equiv V_a(\om,u) = K \sum_{x\in\T_N^d} \tilde\om_x^{(a)} \om_{x+n_1},
$$
where $\tilde\om_x^{(a)}$ is defined in Corollary \ref{Cor:2.6}.
The first step is to replace $V_a$ by its local sample average $V_a^\ell$
defined by
$$
V_a^\ell := K \sum_{x\in\T_N^d} \overleftarrow{(\tilde\om_\cdot^{(a)})}_{x,\ell} 
\overrightarrow{(\om_{\cdot+n_1})}_{x,\ell}, \quad \ell\in\N,
$$
where
$$
\overrightarrow{g}_{x,\ell} := \frac1{|\La_\ell|} \sum_{y\in\La_\ell}g_{x+y},
\quad
\overleftarrow{g}_{x,\ell} := \frac1{|\La_\ell|} \sum_{y\in\La_\ell}g_{x-y},
$$
for functions $g=\{g_x(\eta)\}_{x\in\T_N^d}$ and $\La_\ell = [0,\ell-1]^d\cap\Z^d$.  
Since $\ell$ will be smaller than $N$, one can regard $\La_\ell$ as a subset of $\T_N^d$. 
The reason that we consider both
$\overrightarrow{g}_{x,\ell}$ and $\overleftarrow{g}_{x,\ell}$ is to make
$h_x^{\ell,j}$ defined by \eqref{eq:2.14} satisfy the condition
$h_x^{\ell,j}(\eta^{x,x+e_j})=h_x^{\ell,j}(\eta)$ for any $\eta\in\mathcal X_N$.

\begin{prop}  \label{prop:2.9}
We assume the conditions of Theorem \ref{thm:EstHent} and write $\nu=\nu_{u(\cdot)}$
and $d\mu=fd\nu$ by omitting $t$.  For $\k>0$ small enough, we choose $\ell = N^{\frac1d(1-\k)}$ when $d\ge 2$ and $\ell=N^{\frac12-\k}$ when $d=1$.  
Then, the cost of this replacement is estimated as
\begin{align}  \label{eq:2.11}
\int(V_a-V_a^\ell)f d\nu \le \e_0 N^2 \mathcal{D}_K(\sqrt{f};\nu)
+ C_{\e_0,\k} \left( H(\mu|\nu) + N^{d-1+\k}\right),
\end{align}
for every $\e_0>0$ with some $C_{\e_0,\k}>0$ when $d\ge 2$ and the last
$N^{d-1+\k}$ is replaced by $N^{\frac12+\k}$ when $d=1$.
\end{prop}

The first step for the proof of this proposition is the flow lemma for the telescopic sum.
We call $\Phi = \{\Phi(x,y)\}_{x\sim y: x, y\in G}$  a flow on a finite graph $G$
connecting two probability measures $p$ and $q$ on $G$ if $\Phi(x,y)
= -\Phi(y,x)$ and $\sum_{z\sim x}\Phi(x,z) = p(x)-q(x)$ hold for all $x, y \in G: x\sim y$.
We define a cost of a flow $\Phi$ by
$$
\|\Phi\|^2:=\dfrac{1}{2}\sum_{x\sim y} \Phi(x,y)^2.
$$
The following lemma has been proved in Appendix G of \cite{JM2}.

\begin{lem}[Flow lemma]\label{lem:flow}
For each $\ell\in\N$, let $p_\ell$ be the uniform distribution on $\La_\ell$
and set $q_\ell:=p_\ell* p_\ell$. Then, there exists a flow
$\Phi^\ell$ on $\La_{2\ell-1}$ connecting the Dirac measure $\de_0$ and $q_\ell$ such that
$\|\Phi^\ell\|^2 \le C_d g_d(\ell)$
with some constant $C_d>0$, independent of $\ell$, where
\begin{align*}
g_d(\ell)=
\begin{cases}
\ell, & \text{if $d=1$,}\\
\log \ell, & \text{if $d=2$,}\\
1, & \text{if $d\ge 3$.}
\end{cases}
\end{align*}
\end{lem}

The flow stated in Lemma \ref{lem:flow} is constructed step by step as follows.
For each $k=0,\dots,\ell-1$, we first construct a flow $\Psi^\ell_k$ connecting $p_k$ and $p_{k+1}$
such that $\sup_{x,j}|\Psi^\ell_k(x,x+e_j)|\le c k^{-d}$ with some $c>0$.
Then we can obtain the flow $\Psi^\ell$ connecting $\de_0$ and $p_\ell$
by simply summing up $\Psi^\ell_k$: $\Psi^\ell:=\sum_{k=0}^{\ell-1} \Psi_k^\ell$.
It is not difficult to see that the cost of $\Psi^\ell$ is bounded by $Cg_d(\ell)$.
Finally, we define the flow $\Phi^\ell$ connecting $\de_0$ and $q_\ell$ by
$$
\Phi^\ell(x,x+e_j):=\sum_{z\in \La_{\ell}}\Psi^\ell(x-z,x-z+e_j)p_\ell(z),
$$
whose cost is bounded by $C_dg_d(\ell)$; see \cite{JM2} for more details.


Recall $p_\ell(y)$ defined in Lemma \ref{lem:flow} and
note that $p_\ell$ can be regarded as a probability distribution on $\T_N^d$.
Set $\hat{p}_\ell(y) = p_\ell(-y)$, then we have
\begin{align*}
g*p_\ell &= \sum_{y\in\T_N^d} g_{x-y} p_\ell(y) = \frac1{|\La_\ell|} 
\sum_{y\in \La_\ell}g_{x-y} = \overleftarrow{g}_{x,\ell},
\end{align*}
and similarly $g*\hat p_\ell = \overrightarrow{g}_{x,\ell}$.
Therefore,
\begin{align*}
V_a^\ell 
& = K \sum_{x\in\T_N^d} (\tilde\om_\cdot^{(a)}*p_\ell)_x  (\om_{\cdot+n_1}* \hat p_\ell)_x \\
& = K \sum_{x\in\T_N^d} \left( \sum_{y\in\T_N^d} \tilde\om^{(a)}_y p_\ell(x-y)\right)
\left(\sum_{z\in\T_N^d} \om_{z+n_1} p_\ell(z-x) \right) \\
& = K \sum_{y\in\T_N^d} \tilde\om^{(a)}_y \sum_{z\in\T_N^d} \om_{z+n_1} p_\ell*p_\ell(z-y) \\
& = K \sum_{y\in\T_N^d} \tilde\om^{(a)}_y (\om_{\cdot+n_1} * \hat q_\ell)_y,
\end{align*}
where $q_\ell$ is defined in Lemma \ref{lem:flow} and $\hat q_\ell(y) := q_\ell(-y)$.
Note that supp$\,q_\ell \subset \La_{2\ell-1} = [0,2\ell-2]^d \cap \Z^d$.
Let $\Phi^\ell$ be a flow given in Lemma \ref{lem:flow}.
Accordingly, since $\Phi^\ell$ is a flow connecting $\de_0$ and $q_\ell$,
one can rewrite
\begin{align*}
V_a-V_a^\ell
& = K \sum_{x\in\T_N^d} \tilde\om^{(a)}_x \left\{\om_{x+n_1} - \sum_{y\in\T_N^d}\om_{x-y+n_1} \hat q_\ell(y)\right\} \\
& = K \sum_{x\in\T_N^d} \tilde\om^{(a)}_x \sum_{y\in\T_N^d} \om_{x+y+n_1}(\de_0(y)- q_\ell(y))  \\
& = K \sum_{j=1}^d \sum_{x\in\T_N^d} \tilde\om^{(a)}_x \sum_{y\in\T_N^d} \om_{x+y+n_1} 
\{\Phi^\ell(y,y+ e_j) - \Phi^\ell(y-e_j,y)\} \\
& = K \sum_{j=1}^d \sum_{x\in\T_N^d} \tilde\om^{(a)}_x \sum_{y\in\T_N^d} (\om_{x+y+n_1} - \om_{x+(y+e_j)+n_1} )
\Phi^\ell(y,y+ e_j) \\
& = K \sum_{j=1}^d \sum_{x\in\T_N^d}  \left(\sum_{y\in\T_N^d} \tilde\om^{(a)}_{x-y-n_1} \Phi^\ell(y,y+ e_j) \right) 
(\om_x -\om_{x+e_j} ).
\end{align*}
For the last line, we introduced the change of variables $x+y+n_1 \mapsto x$ for the sum
in $x$. Thus, we have shown
\begin{equation}\label{eq:2.13}
V_a-V_a^\ell = K \sum_{j=1}^d \sum_{x\in \T_N^d} h_x^{\ell,j} (\om_x-\om_{x+e_j}),
\end{equation}
where
\begin{equation}\label{eq:2.14}
h_x^{\ell,j}= \sum_{y\in \La_{2\ell-1}} \tilde\om_{x-y-n_1}^{(a)}\Phi^\ell(y,y+e_j).
\end{equation}
Note that $h_x^{\ell,j}$ satisfies $h_x^{\ell,j}(\eta^{x,x+e_j}) = h_x^{\ell,j}(\eta)$ for any $\eta\in\mathcal X_N$.
Indeed, in \eqref{eq:2.14}, $\tilde\om_{x-y-n_1}^{(a)}(\eta^{x,x+e_j}) \not= 
\tilde\om_{x-y-n_1}^{(a)}(\eta)$
only if $x-y-n_1=x$ or $x+e_j$, namely, $y=-n_1$ or $y=-n_1-e_j$, but
these $y$ are not in $\La_{2\ell-1}$ due to the condition (A1) for $n_1$.

Another lemma we use is the integration by parts formula under the 
Bernoulli measure $\nu_{u(\cdot)}$ with a spatially dependent mean.
We will apply this formula for the function $h=h_x^{\ell,j}$.

\begin{lem}[Integration by parts]\label{lem:IP-0}
Let $\nu=\nu_{u(\cdot)}$ and assume 
$u_-\le u_x, u_y\le u_+$ holds for $x, y\in \T_N^d: |x-y|=1$
with some $0<u_-<u_+<1$.  Let $h=h(\eta)$ be a function
satisfying $h(\eta^{x,y})=h(\eta)$ for any $\eta\in\mathcal X_N$. 
Then, for a probability density $f$ with respect to $\nu$, we have
$$
\int h(\eta_y-\eta_x) f d\nu=\int h(\eta) \eta_x\big( f(\eta^{x,y})-f(\eta)\big) d\nu +R_1,
$$
and the error term $R_1=R_{1,x,y}$ is bounded as
$$
|R_1| \le C |\nabla_{x,y}^1 u| \int|h(\eta)| f d\nu,
$$
with some $C=C_{u_-,u_+}>0$, where $\nabla_{x,y}^1 u =u_x-u_y$.
\end{lem}

\begin{proof}
First we write
$$
\int h(\eta_y-\eta_x) f d\nu=\sum_{\eta\in\mathcal{X}_N}
 h(\eta) (\eta_y-\eta_x) f(\eta)\nu(\eta).
$$
Then, by the change of variables $\eta^{x,y}\mapsto\eta$
and noting the invariance of $h$ under this change, we have
\begin{align*}
\sum_{\eta\in\mathcal{X}_N} h(\eta) \eta_y f(\eta)\nu(\eta)
 = \sum_{\eta\in\mathcal{X}_N} h(\eta) \eta_x f(\eta^{x,y})\nu(\eta^{x,y}).
\end{align*}
To replace the last $\nu(\eta^{x,y})$ by $\nu(\eta)$, we observe
\begin{align*}
\frac{\nu(\eta^{x,y})}{\nu(\eta)} & =
{\bf 1}_{\{\eta_x=1,\eta_y=0\}} \frac{(1-u_x)u_y}{u_x(1-u_y)}
+ {\bf 1}_{\{\eta_x=0,\eta_y=1\}} \frac{u_x(1-u_y)}{(1-u_x)u_y}
+ {\bf 1}_{\{\eta_x=\eta_y\}} \\
& = 1+r_{x,y}(\eta),
\end{align*}
with
\begin{align*}
r_{x,y}(\eta)= {\bf 1}_{\{\eta_x=1,\eta_y=0\}} \frac{u_y-u_x}{u_x(1-u_y)}
+ {\bf 1}_{\{\eta_x=0,\eta_y=1\}} \frac{u_x-u_y}{(1-u_x)u_y}.
\end{align*}
By the condition on $u$, this error is bounded as
$$
|r_{x,y}(\eta)| \le C_0  |\nabla_{x,y}^1 u|, \quad C_0 = C_{u_-,u_+}>0.
$$
These computations are summarized as
\begin{align*}
\int h(\eta_y-\eta_x) f d\nu=&\int h(\eta) \eta_x f(\eta^{x,y}) (1+ r_{x,y}(\eta)) d\nu
- \int h(\eta) \eta_x f(\eta) d\nu \\
=&\int h(\eta) \eta_x\big( f(\eta^{x,y})-f(\eta)\big) d\nu
+ \int h(\eta) \eta_x f(\eta^{x,y}) r_{x,y}(\eta) d\nu.
\end{align*}
For the second term denoted by $R_1$, applying the change of variables
$\eta^{x,y} \mapsto \eta$  again, we have
\begin{align*}
|R_1| & = \left|  \sum_{\eta\in\mathcal{X}_N} 
h(\eta) \eta_y f(\eta) r_{x,y}(\eta^{x,y}) \nu(\eta^{x,y}) \right| \\
& = \left|  \sum_{\eta\in\mathcal{X}_N} h(\eta) \eta_y f(\eta) r_{x,y}(\eta^{x,y}) \big(1+ r_{x,y}(\eta)\big)
 \nu(\eta)\right| \\
& \le C_0  |\nabla_{x,y}^1 u| (1+C_0  |\nabla_{x,y}^1 u| ) \int|h(\eta)| fd\nu
\le C  |\nabla_{x,y}^1 u| \int|h(\eta)| fd\nu,
\end{align*}
since $|\eta_y|\le 1$ and $ |\nabla_{x,y}^1 u|\le 2$.
This completes the proof.
\end{proof}

We apply Lemma \ref{lem:IP-0} to $V_a-V_a^\ell$ given in \eqref{eq:2.13}.
However,  $\om_x= (\eta_x-u_x)/\chi(u_x)$ in \eqref{eq:2.13} depends on
$u_x$ which varies in space.  We need to estimate the error caused by 
this spatial dependence.

\begin{lem} \label{lem:2.13}
{\rm (1)} Assume that $\nu=\nu_{u(\cdot)}$ and $h=h(\eta)$ satisfy the same conditions
as in Lemma \ref{lem:IP-0}.  Then, we have
$$
\int h(\om_y-\om_x) f d\nu=\int h(\eta) \frac{\eta_x}{\chi(u_x)} \big( f(\eta^{x,y})-f(\eta)\big) d\nu +R_2,
$$
and the error term $R_2=R_{2,x,y}$ is bounded as
$$
|R_2| \le C |\nabla_{x,y}^1 u| \int|h(\eta)| f d\nu,
$$
with some $C=C_{u_-,u_+}>0$. \\
{\rm (2)}  In particular, for $h_x^{\ell,j}$ defined in \eqref{eq:2.14}, we have
\begin{equation}  \label{eq:2.15}
\int h_x^{\ell,j}(\om_x- \om_{x+e_j}) f d\nu=-\int h_x^{\ell,j}\frac{\eta_x}{\chi(u_x)} 
\big( f(\eta^{x,x+e_j})-f(\eta)\big) d\nu +R_{2,x,j},
\end{equation}
and
$$
|R_{2,x,j}| \le C |\nabla_{x,x+e_j}^1 u| \int|h_x^{\ell,j}(\eta)| f d\nu.
$$
\end{lem}

\begin{proof}
By the definition of $\om_x$, we have
\begin{align*}
\int h(\om_y-\om_x) f d\nu
& =\int h \left(\frac{\eta_y}{\chi(u_y)}-\frac{\eta_x}{\chi(u_x)}\right) f d\nu
- \int h \left(\frac{u_y}{\chi(u_y)}-\frac{u_x}{\chi(u_x)}\right) f d\nu \\
& =: I_1 - I_2.
\end{align*}
For $I_2$, we have
\begin{align*}
\left|\frac{u_y}{\chi(u_y)}-\frac{u_x}{\chi(u_x)}\right|
& \le \frac1{\chi(u_x)\chi(u_y)} \left\{\chi(u_x) |u_y-u_x|+|u_x||\chi(u_x)-\chi(u_y)|\right\} \\
& \le C |\nabla^1_{x,y}u|.
\end{align*}
On the other hand, $I_1$ can be rewritten as
\begin{align*}
I_1 
& = \int \frac{h}{\chi(u_x)} (\eta_y-\eta_x) f d\nu
+ \int h \left(\frac1{\chi(u_y)}-\frac1{\chi(u_x)}\right) \eta_y f d\nu \\
&=: I_{1,1}+I_{1,2}.
\end{align*}
For $I_{1,1}$, one can apply Lemma \ref{lem:IP-0} to obtain
$$
I_{1,1} = \frac1{\chi(u_x)} \int h(\eta^{x,y}) \eta_x\big( f(\eta^{x,y})-f(\eta)\big) d\nu +\frac1{\chi(u_x)}R_1.
$$
Finally for $I_{1,2}$, observe that
$$
\left|\frac1{\chi(u_y)}-\frac1{\chi(u_x)}\right|
= \frac{|\chi(u_x)-\chi(u_y)|}{\chi(u_x)\chi(u_y)}
\le C|\nabla_{x,y}^1u|.
$$
Therefore, we obtain (1).  Since $h_x^{\ell,j}(\eta^{x,x+e_j}) = h_x^{\ell,j}(\eta)$ for any $\eta\in\mathcal X_N$, 
taking $y=x+e_j$ and changing the sign of both sides,
(2) is immediate from (1).
\end{proof}

We can estimate the first term in the right hand side of \eqref{eq:2.15} by the Dirichlet
form of the Kawasaki part and obtain the next lemma.

\begin{lem}  \label{lem:3.5-b}
Let $\nu=\nu_{u(\cdot)}$ satisfy the condition in Lemma \ref{lem:IP-0}
with $y=x+e_j$.  Then, for every $\b>0$, we have
\begin{equation*}
\int h_x^{\ell,j}(\om_x-\om_{x+e_j}) f d\nu
\le \b \mathcal{D}_{K;x,x+e_j}(\sqrt{f};\nu) + \frac{C}\b
\int (h_x^{\ell,j})^2 f d\nu +R_{2,x,j},
\end{equation*}
with some $C=C_{u_-,u_+}>0$, where
\begin{align*}
\mathcal{D}_{K;x,y}(f;\nu) = \frac14\int_{\mathcal{X}_N} \{f(\eta^{x,y})-f(\eta)\}^2 d\nu,
\end{align*}
is a piece of the Dirichlet form $\mathcal{D}_{K}(f;\nu)$ corresponding to the Kawasaki
part considered on the bond $\{x,y\}: |x-y|=1$ and the error term $R_{2,x,j}$ is
given by Lemma \ref{lem:2.13}.
\end{lem}

\begin{proof}
For simplicity, we write $y$ for $x+e_j$.
By decomposing $f(\eta^{x,y})-f(\eta) = \big( \sqrt{f(\eta^{x,y})}+\sqrt{f(\eta)}\big)
\big( \sqrt{f(\eta^{x,y})}-\sqrt{f(\eta)}\big)$, the first term in the right hand side
of \eqref{eq:2.15} is bounded by
$$
\b \mathcal{D}_{K;x,y}(\sqrt{f};\nu) + \frac{2}{\b \chi(u_x)^2}
\int (h_x^{\ell,j})^2 \{f(\eta^{x,y})+f(\eta)\} d\nu,
$$
for every $\b>0$.  Applying the change of variables $\eta^{x,y}\mapsto\eta$,
the second term of the last expression is equal to and bounded by
\begin{align*}
\frac{2}{\b \chi(u_x)^2} \int (h_x^{\ell,j})^2 (1+\dfrac{\nu(\eta^{x,y})}{\nu(\eta)})f d\nu
\le \frac{C}\b \int (h_x^{\ell,j})^2 f d\nu.
\end{align*}
This shows the conclusion.
\end{proof}

We now give the proof of Proposition \ref{prop:2.9}.
\begin{proof}[Proof of Proposition \ref{prop:2.9}]
By Lemma \ref{lem:3.5-b}, choosing $\b= \e_0 N^2/K$ with $\e_0>0$, we have
\begin{align}  \label{eq:2.16-a}
\int &(V_a-V_a^\ell) f d\nu 
= K \sum_{j=1}^d \sum_{x\in \T_N^d} \int h_x^{\ell,j} (\om_{x+e_j} -\om_x) f d\nu\\
& \le \e_0 N^2 \mathcal{D}_K(\sqrt{f};\nu) + \frac{CK^2}{\e_0 N^2}
\sum_{j=1}^d \sum_{x\in \T_N^d} \int (h_x^{\ell,j})^2 f d\nu
+ K \sum_{j=1}^d \sum_{x\in \T_N^d} R_{2,x,j}.  \notag
\end{align} 
For $R_{2,x,j}$, since $|\nabla_{x,x+e_j}^1u| \le CK/N$ from 
the condition (A4) combined with Proposition \ref{prop:nabla-u} stated below,
estimating $|h_x^{\ell,j}| \le 1+(h_x^{\ell,j})^2$, we have
$$
K |R_{2,x,j}| \le \frac{CK^2}N \int \left(1+(h_x^{\ell,j})^2\right) fd\nu.
$$
Thus, estimating $1/N\le 1$ for the second term of \eqref{eq:2.16-a} (though
this term has a better constant $CK^2/N^2$, the same
term with $CK^2/N$ arises from $K|R_{2,x,j}|$), we obtain
\begin{align}\label{est1}
\int (V_a-V_a^\ell) f d\nu 
\le&  \e_0 N^2 \mathcal{D}_K(\sqrt{f};\nu) + \frac{C_{\e_0} K^2}{N}\sum_{j=1}^d
\sum_{x\in \T_N^d}  \int (h_x^{\ell,j})^2 f d\nu 
 + CK^2N^{d-1}.
\end{align}

We assume without loss of generality that $N/2\ell$ is an integer
and that $n_1=(1,\dots, 1)$ for notational simplicity.
Then, for the second term of the right hand side in \eqref{est1}, 
we first decompose the sum $\sum_{x\in \T_N^d}$ as
$\sum_{y\in \La_{2\ell}}\sum_{z\in (2\ell) \T_N^d}$ regarding $x=z+y$.
Note that the random variables $\{h_{z+y}^{\ell, j}\}_{z\in (2\ell) \T_N^d}$
are independent for each $y\in\La_{2\ell}$. Recall that $d\mu=fd\nu$.
Then, applying the entropy inequality, we have
\begin{align*}
\sum_{x\in \T_N^d}  \int (h_x^{\ell,j})^2 f d\nu 
& \le \frac1\ga \sum_{y\in \La_{2\ell}} \left( H(\mu|\nu) +
\log \int \exp\left\{ \ga \sum_{z \in (2\ell)\T^d_N} (h_{z+y}^{\ell,j})^2 \right\} d\nu\right) \\
& = \frac1\ga (2\ell)^d \left( H(\mu|\nu) + \sum_{z \in (2\ell)\T^d_N} 
\log\int \exp\left\{ \ga (h_{z+y}^{\ell,j})^2 \right\} d\nu \right).
\end{align*}
Now we apply the concentration inequality (see Appendix B of \cite{JM1}) for the last term:

\begin{lem}[Concentration inequality]\label{lem:ci}
Let $\{X_i\}_{i=1}^n$ be independent random variables with values in the
intervals $[a_i,b_i]$.  Set $\bar{X}_i = X_i - E[X_i]$ and $\bar\si^2 = \sum_{i=1}^n (b_i-a_i)^2$.
Then, for every $\ga\in [0,(\bar\si^2)^{-1}]$, we have
$$
\log E\left[\exp \left\{\ga \left(\sum_{i=1}^n \bar{X}_i\right)^2\right\} \right]
\le 2\ga\bar\si^2 \, (\le 2).
$$
\end{lem}

In fact, since $h_x^{\ell,j}$ is a weighted sum of independent
random variables, from this lemma, we have
$$
\log \int e^{\ga(h_x^{\ell, j})^2} d\nu \le 2,
$$
for every $\ga \le C_0/\si^2$, where $C_0$ is a universal constant
and $\si^2$ is the variance of $h_x^{\ell, j}$. 
On the other hand, it follows from the flow lemma that
$\si^2 \le C_d g_d(\ell)$.
Therefore, we have
\begin{align*}
\sum_{x\in \T_N^d}  \int (h_x^{\ell,j})^2 f d\nu 
\le \frac1\ga (2\ell)^d \left( H(\mu|\nu) + 2 (\tfrac{N}{2\ell})^d  \right).
\end{align*}
Thus, taking $\ga^{-1}= (C_dg_d(\ell))/C_0$, we have shown
\begin{align}  \label{eq:Va-1}
\int (V_a-V_a^\ell) f d\nu 
\le \e_0 N^2 \mathcal{D}_K(\sqrt{f};\nu) + 
\frac{CK^2 \ell^dg_d(\ell)}{N} \left( H(\mu|\nu) + \tfrac{N^d}{\ell^d}\right)+ CK^2N^{d-1},
\end{align} 
with some $C=C_{\e_0}>0$. For $\kappa>0$ small enough,
choose $\ell=N^{\frac1d(1-\k)}$ when $d\ge 2$ and
$\ell=N^{\frac12-\k}$ when $d=1$.  Then, recalling $1\le K \le \de(\log N)^{1/2}$
in the condition (A6)$_\de$, when $d\ge 2$,
we have
\begin{align}  \label{eq:Va-2}
\frac{K^2 \ell^dg_d(\ell)}{N}\le \de^2 N^{-\k}g_d(\ell) \log N \le 1,
\quad \tfrac{N^d}{\ell^d}= N^{d-1+\k},
\quad K^2N^{d-1} \le N^{d-1+\k},
\end{align} 
which shows \eqref{eq:2.11}.  When $d=1$,
\begin{align}  \label{eq:Va-3}
\frac{K^2 \ell^2}{N} \le \de^2N^{-2\k} \log N \le 1,
\quad \frac{N}\ell = N^{\frac12+\k}, \quad
K^2N^{d-1} \le N^\k.
\end{align} 
This shows the conclusion for $d=1$ and the proof of Proposition \ref{prop:2.9} is complete.
\end{proof}

\subsection{Estimate on $\int V_a^\ell f d\nu_{u(\cdot)}$}  \label{sec:2.6.3-a}

The next step is to estimate the integral $\int V_a^\ell f d\nu$.
We assume the same conditions as in Proposition \ref{prop:2.9} and therefore
Theorem \ref{thm:EstHent}.  We again decompose the sum $\sum_{x\in \T_N^d}$ as
$\sum_{y\in \La_{2\ell}}\sum_{z\in (2\ell) \T_N^d}$ and then, noting the $(2\ell)$-dependence 
of $\overleftarrow{(\tilde\om_\cdot^{(a)})}_{x,\ell} \overrightarrow{(\om_{\cdot+e})}_{x,\ell}$,
use the entropy inequality, the elementary inequality $ab\le (a^2+b^2)/2$
and the concentration inequality to show
\begin{align*}
\int V_a^\ell f d\nu
& \le \frac{K}\ga \sum_{y\in\La_{2\ell}} \left\{ H(\mu|\nu) +
\sum_{z\in (2\ell)\T_N^d} \log E^{\nu}[e^{\ga \overleftarrow{(\tilde\om_\cdot^{(a)})}_{z+y,\ell} 
\overrightarrow{(\om_{\cdot+e})}_{z+y,\ell}}] \right\}  \\
& \le \frac{K(2\ell)^d}\ga \left\{ H(\mu|\nu) + \frac{N^d}{(2\ell)^d} C_1 \ga \ell^{-d}\right\},
\end{align*}
for $\ga=c\ell^d$ with $c>0$ small enough.
Roughly saying, by the central limit theorem, both $\overleftarrow{(\tilde\om_\cdot^{(a)})}_{x,\ell}$
and $\overrightarrow{(\om_{\cdot+e})}_{x,\ell}$ behave as $C_2\ell^{-d/2}N(0,1)$ in law 
for large $\ell$, respectively, where $N(0,1)$ denotes a Gaussian random variable
with mean $0$ and variance $1$.  This effect is controlled by the concentration
inequality.  When $d\ge 2$, we chose $\ell = N^{\frac1d(1-\k)}$ 
so that we obtain
\begin{equation}  \label{eq:2V-2}
\int V_a^\ell f d\nu\le C_3K \big(H(\mu|\nu) + N^{d-1+\k}\big).
\end{equation}
When $d=1$, we chose $\ell= N^{\frac12-\k}$ so that we obtain \eqref{eq:2V-2}
with $N^{d-1+\k}$ replaced by $N^{\frac12+\k}$.

\subsection{Estimates on three other terms $V_b, V_c, V_1$}  \label{sec:3.3}

Two terms $V_b$ and $V_c$ defined in \eqref{eq:cor26} can be 
treated exactly in a same way as $V_a$  and we have similar results 
to Proposition \ref{prop:2.9} and
\eqref{eq:2V-2} for these two terms.  

The term $V_1$ requires more careful study.  As 
we pointed out at the beginning of this section,
the condition (A4) combined with Proposition \ref{prop:nabla-u} shows that
\begin{equation}  \label{eq:3-NK}
N^2 (u_y(t)-u_x(t))^2 \le CK^2, \qquad t \in [0,T], \;
x,y\in \T_N^d:|x-y|=1.
\end{equation}
Therefore, the front factor behaves like $K^2$ instead of $K$.
Noting this, for the replacement of $V_1$ with $V_1^\ell$, we have a similar 
bound \eqref{eq:Va-1} with $K$ replaced by $K^2$.  However, since
$K\le \de (\log N)^{1/2}$, one can absorb even $K^2$ by the factor $N^\k$ with $\k>0$ as in
\eqref{eq:Va-2} and \eqref{eq:Va-3} (with $K$ replaced by $K^2$).
Thus, the bound \eqref{eq:2.11} in Proposition \ref{prop:2.9} holds also for
$V_1-V_1^\ell$ in place of $V_a-V_a^\ell$.

On the other hand, \eqref{eq:2V-2} should be modified as
\begin{equation}  \label{eq:2V-2-b}
\int V_1^\ell f d\nu\le C_4K^2 \big(H(\mu|\nu) + N^{d-1+\k}\big).
\end{equation}
Note that \eqref{eq:3-NK} holds with $K$ instead of $K^2$ in an averaged sense
in $(t,x)$ as we will see in Lemma \ref{lem:4.4}.  But this is not enough to improve
\eqref{eq:2V-2-b} with $K^2$ to $K$.

\subsection{Completion of the proof of Theorem \ref{thm:EstHent}}  \label{sec:3.4}

Finally, from Proposition \ref{thm:4.2}, Proposition
\ref{prop:2.9} (for $V_a, V_b, V_c, V_1$) and \eqref{eq:2V-2} (for $V_a, V_b, V_c$),
\eqref{eq:2V-2-b} (for $V_1$), choosing $\e_0>0$ small enough such that $4\e_0<2$, 
we obtain
$$
\frac{d}{dt} H(\mu_t|\nu_t) \le C K^2 H(\mu_t|\nu_t) +  O(N^{d-\a}),
$$
with some $0<\a<1$ ($\a=1-\k$) when $d\ge 2$ and $0<\a<1/2$ ($\a=1/2-\k$) when $d=1$.
Thus, Gronwall's inequality shows
$$
H(\mu_t|\nu_t) \le \left( H(\mu_0|\nu_0)+ t O(N^{d-\a}) \right) e^{CK^2t}.
$$
Noting $H(\mu_0|\nu_0) = O(N^{d-\de_0})$ with $\de_0>0$ and $e^{CK^2t} \le N^{Ct\de^2}$ 
from $1\le K\le \de (\log N)^{1/2}$ in the condition (A6)$_\de$, 
this concludes the proof of Theorem \ref{thm:EstHent}, if we choose
$\de>0$ small enough.

\begin{rem}  \label{rem:3.1}
The above argument actually implies  $H(\mu_t^N|\nu_t^N)=O(N^{d-\delta_*})$
for some $\delta_*>0$.  
From Theorem \ref{prop:4.4-2}, the probability in the left hand side of \eqref{eq:thm1.1}
is bounded above by $\mu_t^N(\mathcal A_{N,t}^{\e/2})$ for $N$ sufficiently large,
recall $\mathcal A_{N,t}^{\e}$ defined below \eqref{eq:4.18}.   On the other hand, 
from the proof of Proposition \ref{prop:4.4-1}, there exists a constant
$C_0$, which depends only on $\|\varphi\|_\infty$, such that
$\nu_t^N(\mathcal A_{N,t}^{\e/2})\le e^{-C_0\e^2N^d}$.
These estimates together with the entropy inequality show that
\begin{align*}
P\left( \left|\lan \a^N(t),\fa\ran - \lan \chi_{\Ga_t},\fa\ran \right|>\e\right)
\le \frac{C}{\e^2} N^{-\de_*},
\end{align*}
for $N$ sufficiently large.
This gives the rate of convergence in the limit \eqref{eq:thm1.1}.
\end{rem}

\section{Motion by mean curvature from Glauber-Kawasaki dynamics}
\label{sec:4.4}

The rest is to study the asymptotic behavior as $N\to\infty$ of the solution $u^N(t)$
of the discretized hydrodynamic equation \eqref{eq:HD-discre}, which appears in \eqref{eq:4.18}.
We also give a few estimates on $u^N(t)$ which were already used in Section 3. 

Theorem \ref{prop:4.4-2} formulated below is purely a PDE type result, 
which establishes the sharp interface limit for $u^N(t)$ and leads to the 
motion by mean curvature.  Recall that we assume $d\ge 2$.  A smooth family of 
closed hypersurfaces $\{\Ga_t\}_{t\in [0,T]}$ in $\T^d$ is called the motion 
by mean curvature flow starting from $\Ga_0$, if it satisfies
\begin{align}  \label{eq:4.MMC}
\begin{aligned}
& V = \k, \quad t \in [0,T],\\
& \Ga_t|_{t=0}=\Ga_0,
\end{aligned}
\end{align}
where $V$ is the inward normal velocity of $\Ga_t$ and $\k$ is the mean curvature
of $\Ga_t$ multiplied by $(d-1)$.  It is known that if $\Ga_0$ is a smooth hypersurface without boundary,
then there exists a unique smooth solution of \eqref{eq:4.MMC} starting from $\Ga_0$ on $[0,T]$ with 
some $T>0$; cf.\ Theorem 2.1 of \cite{CHL} and see Section 4 of \cite{F16} for related
references.  In fact, by using the local coordinate $(d,s)$
for $x$ in a tubular neighborhood of $\Ga_0$ where $d=d(x)$ is the signed distance from
$x$ to $\Ga_0$ and $s=s(x)$ is the projection of $x$ on $\Ga_0$, $\Ga_t$ is expressed
as a graph over $\Ga_0$ and represented by $d=d(s,t)$, $s\in \Ga_0$, $t\in [0,T]$, and
the equation \eqref{eq:4.MMC} for $\Ga_t$ can be rewritten as a quasilinear parabolic equation
for $d=d(s,t)$.  A standard theory of quasilinear parabolic equations shows the 
existence and uniqueness of smooth local solution in $t$.  We cite \cite{Be} as an
expository reference for the definitions of mean curvature, motion by mean curvature
flow and Allen-Cahn equation.  As we mentioned above, Section 4 of \cite{F16}
also gives a brief review of these topics.

The limiting behavior of $u^N(t)$ as $N\to\infty$ is given by the following theorem.
Recall that the solution $\{u^N(t,x), x\in \T_N^d\}$ is extended as a step function
$\{u^N(t,v), v\in \T^d\}$ on $\T^d$ as in \eqref{eq:2.uN-m}.

\begin{theorem}  \label{prop:4.4-2}
Under the conditions {\rm (A2), (A3)} and {\rm (A5)}, for $t\in [0,T]$ and $v\notin\Ga_t$,
$u^N(t,v), v\in \T^d$ converges as $N\to\infty$ to $\chi_{\Ga_t}(v)$ 
defined by \eqref{eq:1.chiga} from the hypersurface $\Ga_t$ in $\T^d$ moving
according to the motion by mean curvature \eqref{eq:4.MMC}.
\end{theorem}

Combining the probabilistic result \eqref{eq:4.18} with this PDE type result, we have proved that
$\a^N(t)$ converges to $\chi_{\Ga_t}$ in probability when multiplied by a test function
$\fa\in C^\infty(\T^d)$.  This completes the proof of our main Theorem \ref{thm:4.10}.  

Under the condition (A7)$'$, especially with $u_\pm$ chosen as  
$0<u_-<\a_1< \a_2<u_+<1$, by the comparison theorem (Proposition
\ref{prop:Comparison} below) for the discretized hydrodynamic
equation \eqref{eq:HD-discre} and noting that, if $u^N(0,x)\equiv u_-$ 
(or $u_+$), the solution $u^N(t,x)\equiv u^N(t)$ of  \eqref{eq:HD-discre}
increases in $t$ toward $\a_1$ (or decreases to $\a_2$) by the condition (A2),
the condition $u^N(0,x)\in [u_-,u_+]$ implies the same for $u^N(t,x)$.
In particular, this shows $\chi(u^N(t,x))\ge c>0$ with some $c>0$ for all
$t\in [0,T]$ and $x\in \T_N^d$.

\subsection{Estimates on the solutions of the discretized hydrodynamic equation}
\label{sec:2.5}

We give estimates on the gradients of the solutions $u(t)\equiv u^N(t) =\{u_x(t)\}_{x\in\T_N^d}$ of the
discretized hydrodynamic equation \eqref{eq:HD-discre}.  These were used
to estimate the contribution of  the first term in \eqref{eq:2.int} and also
$R_{2,x,j}$ in \eqref{eq:2.16-a} as we already mentioned.
Let $p^N(t,x,y)$ be the discrete heat kernel corresponding to
$\De^N$ on $\T_N^d$.  Then, we have the following  global estimate in $t$.

\begin{lem} \label{lem:2.1}
There exist $C, c>0$ such that
$$
\|\nabla^N p^N(t,x,y)\|
\le \frac{C}{\sqrt{t}} p^N(ct,x,y), \quad t>0,
$$
where
$
\nabla^N u(x) = \left\{ N(u(x+e_i)-u(x))\right\}_{i=1}^d
$
and $\|\cdot\|$ stands for the standard Euclidean norm of $\R^d$ as we defined before.
\end{lem}

\begin{proof}
Let $p(t,x,y)$ be the heat kernel corresponding to the discrete Laplacian
$\De$ on $\Z^d$.  Then, we have the estimate
$$
\|\nabla p(t,x,y)\|
\le \frac{C}{\sqrt{1\vee t}} p(ct,x,y), \quad t >0, \; x, y \in \Z^d,
$$
with some constants $C, c>0$, independent of $t$ and $x, y$, where $\nabla=\nabla^1$.
For example, see (1.4) in Theorem 1.1 of \cite{DD} which discusses more general case
with random coefficients; see also \cite{SZ}. 
Then, since
$$
p^N(t,x,y) = \sum_{k\in N\Z^d} p(N^2t,x,y+k),
$$
the conclusion follows.
\end{proof}

We have the following $L^\infty$-estimate on the gradients of $u^N$.

\begin{prop}  \label{prop:nabla-u}
For the solution $u^N(t,x)$  of \eqref{eq:HD-discre}, we have the estimate
$$
\|\nabla^N u^N(t,x)\| \le K(C_0+C \sqrt{t}), \quad t >0,
$$
if $\|\nabla^N u^N(0,x)\| \le C_0 K$ holds.
\end{prop}

\begin{proof}
From Duhamel's formula, we have
$$
u^N(t,x) = \sum_{y\in\T_N^d} u^N(0,y)p^N(t,x,y)
+K \int_0^t ds \sum_{y\in\T_N^d} f^N(y,u^N(s)) p^N(t-s,x,y).
$$
By noting $f^N(x,u)$ is bounded and applying 
Lemma \ref{lem:2.1}, we obtain the conclusion.
\end{proof}

It is expected that $\nabla^N u^N$ behaves as $\sqrt{K}$ near the interface 
by the scaling property (see Section 4.2 of \cite{F18} and also as Theorem \ref{thm:3.4}
below suggests) and decays rapidly 
in $K$ far from the interface where $u^N(t,x)$ would be almost flat.
In this sense, the estimate obtained in Proposition \ref{prop:nabla-u}
may not be the best possible.  In a weak sense, one can prove the
behavior $u_x(t)-u_y(t) \sim \sqrt{K}/N$ (instead of $K/N$) for
$x,y: |x-y|=1$ as in the next lemma.

\begin{lem}  \label{lem:4.4}
We have
$$
\frac12\sum_{x\in\T_N^d} u_x(T)^2 + \int_0^T N^2 \sum_{x,y\in\T_N^d:|x-y|=1}
(u_x(t)-u_y(t))^2 dt \le CKN^d T + \frac12 N^d.
$$
\end{lem}
\begin{proof}
By multiplying $u_x(t) \, \big(= u^N(t,x)\big)$ to the both sides of \eqref{eq:HD-discre}
and taking the sum in $x$, we have
\begin{align*}
\frac12 \frac{d}{dt} \left\{\sum_{x\in\T_N^d} u_x(t)^2\right\}
& = \sum_{x\in\T_N^d} u_x(t) \De^N u_x(t) + K \sum_{x\in\T_N^d} u_x(t) f^N(x,u(t)) \\
& \le - N^2 \sum_{x,y\in\T_N^d:|x-y|=1} (u_x(t)-u_y(t))^2 + CK N^d.
\end{align*}
Here, we have used the bound $u_x f^N(x, u) \le C$.
Since $\sum_{x\in\T_N^d} u_x(0)^2 \le N^d$, we have the conclusion.
\end{proof}

\subsection{Proof of Theorem \ref{prop:4.4-2}}  \label{sec:4.2-a}

For the proof of Theorem \ref{prop:4.4-2},
we rely on the comparison argument for the discretized hydrodynamic 
equation \eqref{eq:HD-discre}; cf.\ \cite{F99}, Proposition 4.1 of \cite{F03}, 
Lemma 2.2 of \cite{FI} and Lemma 4.3 of \cite{FO}.

Assume that $f^N(x,u), u=(u_x)_{x\in \T_N^d}\in [0,1]^{\T_N^d}$ has the following property:
If $u=(u_x)_{x\in \T_N^d}, v=(v_x)_{x\in \T_N^d}$ satisfies $u\ge v$ 
(i.e., $u_y\ge v_y$ for all $y\in \T_N^d$) and
$u_x = v_x$ for some $x$, then $f^N(x,u)\ge f^N(x,v)$ holds.  Note that
$f^N(x,u)$ given in \eqref{eq:f^N} with $c^\pm(\eta)$ of the form 
\eqref{eq:1.cpm} has this property, since $c_x^+(u)$ is increasing
and $c_x^-(u)$ is decreasing in $u$ in this partial order by the condition (A3).

\begin{prop}  \label{prop:Comparison}
Let $u^\pm(t,x)$ be super and sub solutions of
\begin{equation}  \label{eq:DHD}
\partial_t u^\pm(t,x) = \De^N u^\pm(t,x) + Kf^N(x,u^\pm),\quad x\in \T_N^d.
\end{equation}
Namely, $u^+$ satisfies \eqref{eq:DHD} with \lq\lq$\ge$'', while $u^-$ satisfies it with
\lq\lq$\le$'' instead of the equality. 
If $u^-(0) \le u^+(0)$, then $u^-(t) \le u^+(t)$ holds for all $t>0$.  In particular, 
one can take the solution of \eqref{eq:DHD} as $u^+(t,x)$ or $u^-(t,x)$.
\end{prop}

\begin{proof}
Assume that $u^+(t)\ge u^-(t)$ and $u^-(t,x)=u^+(t,x)$ hold at some $(t,x)$.
Since $u^\pm$ are super and sub solutions of \eqref{eq:DHD}, we have
\begin{align*}
\partial_t (u^+-u^-)(t,x) 
\ge \De^N (u^+-u^-)(t,x) + K \big(f^N(x,u^+)- f^N(x,u^-)\big).
\end{align*}
On the other hand, noting that
\begin{align*}
\De^N (u^+-u^-)(t,x) & = N^2\sum_{\pm e_i} \big( (u^+-u^-)(t,x\pm e_i)
- (u^+-u^-)(t,x) \big)  \\
& = N^2\sum_{\pm e_i} (u^+-u^-)(t,x\pm e_i) \ge 0,
\end{align*}
and that $f^N(x,u^+)- f^N(x,u^-)\ge0$ by the assumption, we have $\partial_t (u^+-u^-)(t,x) \ge0$.
This shows that $u^-(t)$ can not exceed $u^+(t)$ for all $t>0$.
\end{proof}

For $\de\in\R$ with $|\de|$ sufficiently small, one can find a traveling wave
solution $U=U(z;\de), z\in \R$, which is increasing in $z$ and its speed 
$c=c(\de)$ by solving an ordinary differential equation:
\begin{align*}
&U'' + cU' + \{f(U)+\de\}=0, \quad z \in \R, \\
& U(\pm\infty) = U_\pm^*(\de),
\end{align*}
where $U_-^*(\de)<U_+^*(\de)$ are two stable solutions of $f(U)+\de=0$.
Note that $U_-^*(0)=\a_1, U_+^*(0)=\a_2$ and $c(0)=0$.  The solution 
$U(z;\de)$ is unique up to a translation and one can choose $U(z;\de)$
satisfying $U_\de(z;\de)\ge 0$; see \cite{CHL}, p.1288.
Note also that the traveling wave solution  $U$ is associated with the 
one-dimensional version of the reaction-diffusion equation and not 
with the discrete equation \eqref{eq:DHD}.  Indeed, $u(t,z) := U(z-ct)$
solves the equation
$$
\partial_t u = \partial_z^2 u +f(u) +\de, \quad z\in \R,
$$
which is a one-dimensional version of \eqref{eq:RD} or
\eqref{eq:RD-2} with $K=1$ considered on the whole line $\R$
in place of $\T^d$ with $f$ replaced by $f+\de$ and connecting
two stable solutions $U_\pm^*(\de)$ at $z=\pm\infty$.

Let $\Ga_t, t\in [0,T]$ be the motion of smooth hypersurfaces in $\T^d$
determined by \eqref{eq:4.MMC}.
Let $\tilde d(t,v), t\in [0,T], v\in \T^d$ be the signed distance
function from $v$ to $\Ga_t$, and similarly to \cite{CHL}, p.1289,
let $d(t,v)$ be a smooth modification of $\tilde d$ such that
\begin{align*}
& d= \tilde d, & & \text{ if }\; |\tilde d(t,v) | < d_0, \\ 
& d_0< |d| \le 2d_0 \; \text{ and }\; d \tilde d > 0, 
& & \text{ if } \; d_0 \le |\tilde d(t,v) | < 2 d_0, \\ 
& |d| = 2d_0 \; \text{ and }\; d \tilde d > 0, 
& & \text{ if } \; |\tilde d(t,v) | \ge 2 d_0,
\end{align*}
where $d_0>0$ is taken such that $\tilde d(t,v)$ is smooth in the domain
$\{(t,v); |\tilde d(t,v) | < 2 d_0, t\in [0,T]\}$.
We define two functions
$\rho_{\pm}(t,v) \equiv \rho_{\pm}^K(t,v)$ by
$$
\rho_\pm(t,v) = U\left( K^{1/2}(d(t,v) \pm K^{-a} e^{m_2t}); \pm K^{-1}
m_3 e^{m_2t}\right), \quad a, m_2, m_3 >0.
$$
Applying Proposition \ref{prop:Comparison} and repeating computations of Lemma 3.4 in
\cite{CHL}, we have the following theorem for $u^N(t,v)$ defined by 
\eqref{eq:2.uN-m} from the solution $u^N(t,x)$ of \eqref{eq:HD-discre}.
The functions $\rho_\pm(t,v)$ describe the sharp transition of $u^N(t,\cdot)$
and change their values quickly from one side to the other crossing 
the interface $\Ga_t$ to the normal direction.

\begin{theorem}  \label{thm:3.4}
We assume the conditions {\rm (A2), (A3)} and {\rm (A5)}, in particular
 $\{\Ga_t\}_{t\in [0,T]}$ is smooth and the following inequality \eqref{eq:3.3-com}
holds at $t=0$.  The condition on $K$ can be relaxed and we assume
$K \le CN^{2/3}$ for $K=K(N) \to \infty$.
Then, taking $m_2, m_3>0$ large enough, there exists $N_0\in \N$ such that
\begin{equation}  \label{eq:3.3-com}
\rho_-(t,v) \le u^N(t,v) \le \rho_+(t,v),
\end{equation}
holds for every $a>1/2$, $t\in [0,T]$, $v= x/N, x \in \T_N^d$ and $N\ge N_0$.
\end{theorem}

\begin{proof}
Let us show that
$$
L^{N,K} \rho_+ := \frac{\partial \rho_+}{\partial t} - \De^N\rho_+ - K f^N(x,\rho_+) \ge 0,
$$
for every $N\ge N_0$ with some $N_0\in \N$.  We decompose
\begin{align}  \label{eq:4.LNK}
& L^{N,K} \rho_+ = L^K \rho_+ 
+( \De\rho_+ - \De^N\rho_+) + K(f(\rho_+(\tfrac{x}N)) -  f^N(x,\rho_+)),
\intertext{where $\De$ is the continuous Laplacian on $\T^d$ and}
& L^K \rho_+ = \frac{\partial \rho_+}{\partial t} -\De\rho_+ -K f(\rho_+(v)).
\label{eq:4.5-a}
\end{align}
 
The term $L^K \rho_+$ can be treated as in \cite{CHL}, from the bottom of p.1291 to
p.1293.  Note that $\e^{-2}$ in their paper corresponds to $K$ here, and
they treated the case with a non-local term, which we don't have.
Since we can extend $m_1 \e e^{m_2 t}$ in the definition of super and sub solutions
in their paper to $K^{-a}e^{m_2t}$ (i.e., we can take $K^{-a}$ instead of $m_1\e$)
for every $a>0$, we briefly repeat their argument by adjusting it to our setting.
The case with noise term is discussed by \cite{F99}, pp.412-413.

In fact, $L^K \rho_+$ can be decomposed as
$$
L^K \rho_+ = T_1+T_2 + T_3 + m_3e^{m_2 t},
$$
where
\begin{align*}
T_1 & = \left\{ \frac{\partial d}{\partial t} + K^{-a}m_2e^{m_2t}-\De d
+ K^{1/2} c(K^{-1}m_3e^{m_2t}) \right\} K^{1/2} U_z, \\
T_2 & = \left\{ 1- |\nabla d|^2 \right\} K U_{zz},  \\
T_3 & = K^{-1}m_3m_2e^{m_2t}U_\de,
\end{align*}
by just writing $K^{-a}$ instead of $m_1\e$ (i.e., here $m_1 = K^{1/2-a} \searrow 0$)
in \cite{CHL}, p.1292 noting that $W(\nu,\de) = c(\de)$, $U(z;\nu,\de) = U(z;\de)$ 
and $C_2=0$ in our setting.  Repeating their arguments, one can show that, if $m_2$ is 
large enough compared with $m_3$, $T_1, T_2 \ge -C$ hold for some $C>0$,
and $T_3\ge 0$ since $U_\de \ge 0$.  Therefore, we obtain
\begin{equation}  \label{eq:4.LK-m}
L^K\rho_+ \ge m_3e^{m_2 t} -2C.
\end{equation}

For the rest in \eqref{eq:4.LNK}, since $d(t,v)$ and $U(z)$ so that
$\rho_\pm$ are smooth in $v$, we have
\begin{align*}
& \De\rho_+(t,\tfrac{x}N) - \De^N\rho_+(t,\tfrac{x}N)  = O\Big( N^2 \big(\tfrac{K^{1/2}}N\big)^3 \Big)
= O\Big( \tfrac{K^{3/2}}N \Big), \\
& K \big(f(\rho_+(t,\tfrac{x}N)) -  f^N(x,\rho_+(t,\cdot))\big) = K \cdot O\Big(\tfrac{K^{1/2}}N \Big).
\end{align*}
The first one follows from Taylor expansion for $\De^N\rho_+$ up to the
third order term, while the second one follows by taking the expansion up to the first order
term.  Therefore, if $K=O(N^{2/3})$, these terms stay bounded in $N$ and
are absorbed by $L^K\rho_+$ estimated in \eqref{eq:4.LK-m} with $m_3$ chosen large
enough.  Thus, we obtain $L^{N,K}\rho_+\ge 0$.
The lower bound by $\rho_-(t,v)$ is shown similarly.
\end{proof}

Theorem \ref{prop:4.4-2} readily follows from Theorem \ref{thm:3.4}
by noting that we have from the definitions of $\rho_\pm(t,v)=\rho^K_\pm(t,v)$,
\begin{align}  \label{eq:4.5-A}
\lim_{K\to\infty}\rho^K_\pm(t,v)=U(\pm\infty;0)=
\begin{cases}
\a_1, & \text{if $\tilde d(t,v)<0$},\\
\a_2, & \text{if $\tilde d(t,v)>0$},
\end{cases}
\end{align}
for $t\in[0,T]$ and $v\notin\Gamma_t$.

\begin{rem}
The choice $\pm K^{-1}m_3e^{m_2t}$ in the definition
of the super and sub solutions is the best.  In fact, in stead of $K^{-1}$,
if we take $K^{\b-1}$ with $\b>0$, then we may consider $m_3=K^\b$
but this diverges so that $m_2$ also must diverge.   On the other hand,
if $\b<0$, as the above proof shows, we don't have a good control.
\end{rem}

\end{document}